\documentclass[reqno,12pt]{amsart}
\usepackage{amsmath, amsthm, amssymb,geometry,xcolor,hyperref,enumerate}
\newtheorem{thm}{Theorem}
\newtheorem{prop}[thm]{Proposition}
\newtheorem{lem}[thm]{Lemma}
\newtheorem{cor}[thm]{Corollary}

\newtheorem{rem}[thm]{Remark}

\newtheorem{defn}[thm]{Definition}
\newcommand{\al}{\alpha}
\newcommand{\be}{\beta}

\newcommand{\vep}{\varepsilon}

\newcommand{\RR}{\mathbb{R}}

\newcommand{\sM}{\mathcal{M}}

\newcommand{\sK}{\mathcal{K}}

\newcommand{\sN}{\mathcal{N}}

\numberwithin{equation}{section}
\counterwithin{thm}{section}
\hypersetup{
	colorlinks=true,
	linkcolor=blue,
	citecolor=red,
	urlcolor=magenta
}
\allowdisplaybreaks

	\title[Finiteness and rigidity of optimal destabilizing subvarieties]{Finiteness and rigidity of optimal destabilizing subvarieties for the J-equation} %curves and divisors}
\author{P. Sivaram and Zakarias Sj\"Ostr\"om Dyrefelt}

\address{P. Sivaram \\  Institut for Matematik, Aarhus University, Ny Munkegade 118, 8000, Aarhus C, Denmark. }
\email{sivaram@math.au.dk}

\address{Zakarias Sj\"ostr\"om Dyrefelt \\ Institut for Matematik, Aarhus University, Ny Munkegade 118, 8000, Aarhus C, Denmark. }
\email{dyrefelt@math.au.dk}

\begin{document}

\begin{abstract}
%We prove finiteness of optimally destabilizing subvarieties under weaker conditions than previously known. 
We prove that when the J-equation does not admit a smooth solution, there is always only a finite number of obstructing optimally destabilizing curves and divisors on compact Kähler manifolds of any dimension. For threefolds, this establishes the unconditional finiteness of all optimally destabilizing subvarieties. We moreover characterize the union of these optimal subvarieties, proving that they are contained in the locus where a natural continuity path loses local smooth compactness. Finally, to address subvarieties that are neither curves nor divisors, we prove a rigidity criterion which prevents optimal cycles from moving in compact one-parameter families, with further consequences for their deformation and symmetry properties.
%with consequences for Chow rigidity and actions of the connected automorphism group. 
\iffalse
We study subvarieties which attain the stability threshold for the J-equation. At the semistable boundary, we use Chen's mass-concentration current and a partial regularization away from a fixed Siu upper-level set to show that the optimal destabilizing locus is contained in one fixed proper analytic subset. As a consequence, optimal destabilizing curves and divisors are finite on compact K\"ahler manifolds of arbitrary dimension. Combined with modified-K\"ahler assumptions only in the intermediate dimensions, this gives a finite criterion for the J-equation and, in particular, finiteness of all optimal destabilizers on threefolds. We also show that every optimal subvariety lies in the locus where a natural continuity path loses local smooth compactness. Finally, we prove a nef-threshold criterion which prevents optimal cycles from moving in compact one-parameter families, with consequences for Chow rigidity and connected automorphism actions.
\fi 
\end{abstract}

\maketitle

\section{Introduction}

\noindent  Let \(X\) be a compact K\"ahler manifold of dimension \(n \geq 2\), and let \(\alpha,\beta \in H^{1,1}(X,\mathbb R)\) be K\"ahler classes represented by K\"ahler forms \(\omega\) and \(\chi\) respectively. The J-equation  was first introduced by Donaldson \cite{Don99}, and asks for a K\"ahler metric
\[
\omega_\varphi=\omega+\sqrt{-1}\partial\bar\partial\varphi\in\alpha
\]
satisfying
\[
\operatorname{tr}_{\omega_\varphi}\chi
=
c_{\alpha,\beta},
\qquad
c_{\alpha,\beta}
:=
n\frac{\beta\cdot\alpha^{n-1}}{\alpha^n}.
\]

\noindent It has been a main source of sufficient existence criteria for constant scalar curvature K\"ahler metrics (see e.g. \cite{Che00,So20, SD5}) and it also occurs as the small volume-limit of the deformed Hermitian-Yang-Mills equation. It is a distinguished special case of the generalized Monga-Amp\`ere and Z-critical equations (see e.g. \cite{Pingali, DP20, DMS}).
%It is also a prominent special case of Pingali's generalized Monge-Amp\`ere equation. 

%Its parabolic counterpart, the J-flow, was introduced by Chen as a natural method for finding a solution \cite{Che04}.
The smooth solvability theory of the J-equation is by now well understood, thanks to the work of many mathematicians. Weinkove and Song--Weinkove identified the central role of the strict cone, or subsolution, condition and established convergence of the J-flow under these conditions \cite{Wei04,Wei06,SW08}. Restricting the cone condition to subvarieties leads to a collection of intersection inequalities. Lejmi--Sz\'ekelyhidi proposed these inequalities as a numerical criterion for solvability \cite{LS15}, Collins--Sz\'ekelyhidi developed the corresponding coercivity and toric theory \cite{CS17}, and Gao Chen proved that solvability is equivalent to a uniform form of the numerical criterion \cite{Che21}. The complete Lejmi--Szeklyhidi conjecture was moreover proved on all K\"ahler manifolds by Jian song in \cite{So20}.

The boundary and unstable regimes are less rigid. At the boundary, all the expected numerical inequalities remain nonnegative, but equality may occur on proper subvarieties. In the unstable regime, different subvarieties may compete to give the strongest obstruction. On K\"ahler surfaces, the J-equation can be related to a complex Monge--Amp\`ere equation, and the degeneration of the J-flow has been studied in considerable detail \cite{FLSW,KSD1,Mur26}. More recently, Datar--Mete--Song introduced a minimal-slope and bubbling picture for complex Hessian equations \cite{DMS26}, Khalid--Sj\"ostr\"om Dyrefelt developed wall--chamber decompositions and finiteness results under positivity hypotheses \cite{KSD2}, and Fu established the minimal-slope characterization of J-semistability \cite{Fu}.

These developments suggest studying not only whether the J-equation is solvable, but also the geometry of the subvarieties which most strongly obstruct solvability, in the sense of achieving the inf in the definition of the stability threshold
\[
\Gamma^{\mathrm{pp}}_\beta(\alpha)
=
\inf_{\substack{V\subsetneq X, \; \dim V = p, \\ p = 1,2,\dots, n-1}}
\frac{
\int_V c - p\beta\wedge\alpha^{p-1}
}{
(n-p) {\int_V\alpha^p}
}
\]
where 
\[
c
=
n\frac{\int_X\beta\wedge\alpha^{n-1}}
{\int_X\alpha^n}.
\]
It measures the optimal numerical margin in the subvariety inequalities and we say that a subvariety is \emph{destabilizing} if the infimum is non-positive, and \emph{optimally destabilizing} if moreover it attains the infimum. At the semistable boundary, i.e. when $\Gamma^{\mathrm{pp}}_\beta(\alpha) = 0$, these are precisely the subvarieties on which the corresponding intersection inequality becomes an equality, thanks to a fundamental linearity property observed in \cite{ZD20}.

In a program started in \cite{KSD1,KSD2} it was studied to what extent there can there be infinitely many destabilizing subvarieties, proving finiteness under certain positive conditions of initial data, while noting that finiteness is not always guaranteed. The purpose of this paper is to investigate three related questions. First, to what extent can there be infinitely many optimal destabilizers? Second, do optimal destabilizing subvarieties govern the locus where a continuity method or the J-flow loses regularity? Third, are optimal destabilizing cycles rigid, or can they move in positive-dimensional families? Our results relate optimal destabilization to loss of local smooth compactness, establish finiteness statements in arbitrary dimension, and give an explicit numerical criterion excluding moving families of optimal cycles.

\subsection{A fixed analytic set and degeneration along optimal destabilizers} In \cite{KSD1,KSD2} one of the main ingredients in the proof was to show that the set of (optimally) destabilizing subvarieties is included in the negative part of the Zariski decomposition of a suitable big class (see \cite{Boucksomthesis}), which is then a proper analytic set and has finitely many irreducible components, providing a starting point for a proof by induction on dimension. A basic difficulty in studying optimal destabilizers is that unless one imposes a condition such as that in \cite{KSD2}, it is a priori not clear whether the union of optimal destabilizers are contained in some proper analytic set. In this paper the main idea is to approach this question from a more analytic perspective. Our first result places the entire optimal locus inside one fixed analytic subset given by a Siu upper-level set of an auxiliary mass-concentration current. 

\begin{thm} \label{thm:intro-analytic-envelope} 
    Assume that $\Gamma^{pp}_\be(\al)\leq 0$. Then there exist a K\"ahler current $T_{\mathrm{mc}}\in\alpha $ and a number $\delta_0>0$ such that \[ \mathcal D^{\mathrm{opt}}_{\alpha,\beta} \subset E_{\delta_0}(T_{\mathrm{mc}}) := \{x\in X:\nu(T_{\mathrm{mc}},x)\geq\delta_0\}, \] where $$\mathcal{D}^{\mathrm{opt}}_{\alpha,\beta}=\bigcup_{V\in \mathrm{Des}_{\al,\be}^{\mathrm{opt}}(X)}V$$ is the union of all optimally destabilizing subvarieties with respect to $(X,\alpha,\beta)$. In other words, the optimal destabilizing locus is contained in one fixed proper analytic subset of $X$. 
\end{thm} 

\noindent The analytic set in Theorem~\ref{thm:intro-analytic-envelope} is auxiliary and need not coincide with the optimal locus. Nevertheless, its analyticity is enough to prove finiteness of optimal divisors. Together with a separate non-K\"ahler-locus argument for curves and the methods of \cite{KSD2} in the intermediate dimensions, it gives the finiteness results stated below. There is also a direct relationship with degeneration of the continuity method. Consider the affine path \[ \beta_t = (1-t)\frac{c}{n}\alpha+t\beta \] and let $\tau$ be its first boundary parameter. The optimal subvarieties are unchanged along this path for $t>0$. If $\omega_t\in\alpha$ solves the corresponding J-equation for $t<\tau$, then local smooth compactness must fail at every point of every optimal destabilizing subvariety. 

\begin{thm}\label{prop:intro-forced-nonregularity} 
Let $\mathcal R_\tau$ be the set of points admitting a neighborhood on which some sequence $\omega_{t_j}$, $t_j\uparrow\tau$, converges smoothly. Then \[ \mathcal D^{\mathrm{opt}}_{\alpha,\beta} \subset X\setminus\mathcal R_\tau. \] 
\end{thm} 
\noindent No analyticity is asserted for $X\setminus\mathcal R_\tau$. Thus the two sets play complementary roles: \( E_{\delta_0}(T_{\mathrm{mc}}) \) is a fixed analytic set used to prove finiteness, whereas \( X\setminus\mathcal R_\tau \) records actual degeneration of the continuity path.

\subsection{Finiteness of the set of optimally destabilizing curves and divisors}

Our next main result shows that there are always only finitely many optimally destabilizing curves and divisors on any $(X,\alpha,\beta)$, no matter the dimension of $X$ and without positivity assumptions on $c\alpha - p\beta$ as in \cite{KSD2}. 

\begin{thm}\label{thm:finitely many curves and divisors}
	Let $X$ be any compact K\"ahler manifold and $\alpha, \beta$ two K\"ahler classes. Then the set of optimal destabilizing curves and divisors for $(X,\alpha,\beta)$ is finite. 
\end{thm}

\noindent On threefolds and proper irreducible subvariety is either a curve or divisor, so this recovers, by an independent method, the finiteness part of Liu’s recent threefold theorem \cite{Liu1} which was made available as we were finishing up this project. It also improves on the main result of \cite{KSD2} to test only for subvarieties of intermediary dimensions $p = 2,\dots, n-2$: % When $n \leq 3$ this gives:

\begin{cor}\label{cor:all dim}
Let $X$ be a compact K\"ahler manifold with K\"ahler classes $\alpha$ and $\beta$ such that $$c\alpha - p \beta \in \mathcal M_{p+1} \mathcal K$$ for each $p = 2,\dots,n-2$. Then the set of optimally destabilizing subvarieties for $(X,\alpha,\beta)$ is finite. 
\end{cor}
\noindent We can even be slightly more precise: The set of (not necessarily smooth) $p-1$-dimensional optimal destabilizing subvarieties inside a smooth $p$-dimensional subvariety, is always finite.  

Moreover, if $n = 3$ we have
\begin{cor} \label{cor:three dim}
On any K\"ahler threefold the set of optimally destabilizing subvarieties is finite.
\end{cor}

\noindent Given the above results, what remains is understanding optimally destabilizing subvarieties of \emph{intermediate dimension} $2 \leq p \leq n-2$. In this direction we are not yet able to prove the expected finiteness, but instead we establish that such optimal destabilizers cannot move in continuous families.

\subsection{Rigidity and automorphism invariance for optimally destabilizing subvarieties of dimension $2 \leq p \leq n-2$}
\label{subse:nef threshold}

The threshold $\Gamma^{\mathrm{pp}}_\beta(\alpha)$ is defined by optimizing over all proper subvarieties of $X$. Even when the infimum is attained, it is therefore not a priori clear that an optimal subvariety is geometrically distinguished: the members of a positive-dimensional family of cycles have the same numerical class, and hence the same slope. The next theorem shows that a quantitative amount of positivity of $\beta$ in the direction $-\alpha$ rules out this possibility and forces optimal cycles to be rigid. Although the numerical hypothesis may at first appear technical, it is not merely formal. At the J-semistable boundary $\Gamma^{\mathrm{pp}}_\beta(\alpha)=0$, it is automatic for curves, since $$s_\alpha(\beta) := \sup\{s\in\mathbb R:\beta-s\alpha\ \text{is nef}\}$$ satisfies $s_{\alpha}(\beta) > 0$. More generally, the single inequality \( s_\alpha(\beta)>\frac{1}{6}c \) implies the hypothesis of the following theorem simultaneously for every $1\leq p\leq n-1$, since \[ \frac{p-1}{p(p+1)}\leq\frac{1}{6}. \] 
Thus the theorem below applies automatically to all optimal curves at the boundary, including the familiar surface boundary configurations, and under a simple uniform condition it applies to optimal cycles in every dimension. 

\begin{thm}\label{thm:threshold-rigidity} 
Let $X$ be a compact K\"ahler manifold of dimension $n$, and let $\alpha,\beta\in H^{1,1}(X,\mathbb R)$ be K\"ahler classes. Set \( c:=\mu_{\alpha,\beta}(X)\) and consider the nef threshold \[ s_\alpha(\beta) := \sup\{s\in\mathbb R:\beta-s\alpha\ \text{is nef}\}. \] Moreover, fix an integer $1\leq p\leq n-1$ and assume that \[ s_\alpha(\beta)> \Gamma^{\mathrm{pp}}_\beta(\alpha) + \frac{p-1}{p(p+1)} \bigl(c-n\Gamma^{\mathrm{pp}}_\beta(\alpha)\bigr). \] Then the $p$-dimensional optimal subvarieties of $X$ cannot move in a nonconstant compact one-parameter family. More precisely, there is no nonconstant analytic family $\{V_t\}_{t\in T}$ over a connected compact complex curve $T$ whose general member is an irreducible $p$-dimensional subvariety satisfying \[ \frac{c-\mu_{\alpha,\beta}(V_t)}{n-p} = \Gamma^{\mathrm{pp}}_\beta(\alpha). \] 
\end{thm} 

\begin{rem}
The nef-threshold condition above is genuinely quantitative, and its
range is constrained by a simple universal bound.  Indeed, at the
semistable boundary one always has
\[
s_\alpha(\beta)\leq \frac{c}{n},
\]
since $\beta-s\alpha$ nef implies
\[
0\leq (\beta-s\alpha)\cdot\alpha^{n-1}
=
\left(\frac{c}{n}-s\right)\alpha^n.
\]
Thus the condition
\[
s_\alpha(\beta)>
\frac{p-1}{p(p+1)}c
\]
is compatible with this bound whenever
\[
n(p-1)<p(p+1).
\]
In particular, it is automatic for $p=1$, has a nonempty numerical
range for $p=n-2,n-1$ in every dimension, and has such a range for
every $1\leq p\leq n-1$ when $n\leq5$.  In dimensions $n\leq5$, the
single convenient condition
\[
s_\alpha(\beta)>\frac{c}{6}
\]
therefore implies the rigidity hypothesis simultaneously in every
dimension.  In higher dimensions the criterion is more naturally
viewed dimension by dimension.
\end{rem}

\noindent The theorem is naturally formulated in terms of compact one-parameter families, since these are the families to which the sweep-out argument applies. In the projective setting, however, it has a stronger consequence. The Chow variety is projective, and hence any positive-dimensional deformation locus through a given cycle contains a complete curve. The absence of compact one-parameter families therefore becomes genuine isolation of the optimal cycle. Since a connected algebraic subgroup of $\mathrm{Aut}^0(X)$ acts trivially on cohomology, and hence preserves all the slopes appearing above, the same observation also turns rigidity into invariance under connected group actions. 

\begin{cor}\label{cor:projective-rigidity} 
Assume, in addition, that $X$ is projective, and let $1\leq p\leq n-1$ satisfy the hypothesis of Theorem~\ref{thm:threshold-rigidity}. Then every $p$-dimensional optimal subvariety is isolated in the Chow variety among irreducible reduced $p$-dimensional cycles. In particular, if $G\subset Aut^0(X)$ is a connected algebraic subgroup, then every such optimal subvariety is $G$-invariant. 
\end{cor}

\noindent The group-action conclusion becomes particularly useful when combined with an attainment theorem for $\Gamma^{\mathrm{pp}}_\beta(\alpha)$. Suppose, for example, that \cite{ZD20} guarantees that a nontrivial boundary or unstable pair under consideration produces a proper optimal destabilizing subvariety. If a connected algebraic group $G\subset Aut^0(X)$ moves every proper subvariety, while the hypothesis of Theorem~\ref{thm:threshold-rigidity} holds in every possible dimension, then such an optimal destabilizer would have to be both moved by $G$ and fixed by $G$, which is impossible. Hence the corresponding boundary or unstable configuration cannot occur. Even when invariant proper subvarieties do exist, the conclusion is still restrictive. On a projective toric manifold, an optimal subvariety satisfying the hypothesis of Theorem~\ref{thm:threshold-rigidity} must be invariant under the algebraic torus, and is therefore one of the finitely many closures of torus orbits. On an abelian variety, the translation group acts trivially on cohomology, while no proper nonempty subvariety is invariant under every translation. Hence, under the same numerical hypothesis, no proper optimal destabilizing subvariety can occur.

\subsection{Comparison with the work of J. Liu} While this paper was in the final stages of being written up, two papers of J. Liu appeared on arXiv \cite{Liu1,Liu2}. The first of these papers, \cite{Liu1}, gives a rather complete picture on threefolds and in several directions goes further than our Corollary \ref{cor:three dim}. The second paper treats higher dimensions, but conditionally under certain technical conditions. In this regard our Theorem \ref{thm:finitely many curves and divisors} and Corollary \ref{cor:all dim} go further and in a different direction in dimension $n \geq 4$. Similarly, in the study of the singular behaviour of solutions to the J-equation Theorem \ref{prop:intro-forced-nonregularity} was proven in a stronger form by Liu on threefolds, but in higher dimensions our theorem is not covered by their work. The material of section \ref{sec:rigidity} is not covered by Liu. 

\subsection{Disclaimer about the use of AI} Sections 1–4 are human work, with the exception of AI assistance in drafting parts of the exposition, and in finalizing our proofs of Lemmas 4.1 and 4.2. Section 5 was developed with the assistance of ChatGPT 5.5 and 5.6 using standard prompting, without the use of agents.

\subsection{Acknowledgements} This work was supported by a Villum Young Investigator Grant from the Villum Foundation, project no. 60786.

\bigskip 
\section{Preliminaries}
\label{sec:preliminaries}
\noindent In this section we introduce our notation and recall some basic definitions which we use in the upcoming sections.
%We also collect a few conventions on cycles and their parameter spaces. 
All
subvarieties are assumed irreducible and reduced unless explicitly stated
otherwise.

 %We also use the stability of the K\"ahler property under modifications, in the form proved by Varouchas \cite{VarouchasKahler}. 
 %We shall distinguish carefully between the unconditional reverse Khovanskii--Teissier estimates and the stronger one-Rayleigh inequality introduced below, which is an additional hypothesis and does not hold for arbitrary nef triples, see \cite{HuXiao}.

\subsection{Modified K\"ahler classes and positivity conditions}
    We briefly introduce some of the basic definitions, which also has been introduced or studied in \cite{Bou02,Wu22,KSD2}, we refer the reader to the prior mentioned references for more details.
\begin{defn}
    The {\em minimal multiplicity} at $x\in X$ of the pseudo-effective class $\eta\in H^{1,1}(X,\RR)$  is defined as $\nu(\eta,x):=\sup_{\vep>0}\nu(T_{min,\vep},x) $, where $T_{min,\vep}\in\eta$ is a current with minimal singularities which satisfies $T_{\min,\vep}\geq-\vep\omega$ in the above definition and $\nu(T_{\min,\vep},x)$ is the Lelong number of  at $x$. 
    
    When $Z$ is an irreducible analytic subset, we define the generic minimal
multiplicity of $\eta$ along $Z$ as $\nu(\eta,Z):=\inf_{x\in Z}\nu(\eta,x)$.
\end{defn}

An important point to note here is that the definition is independent of the choice of the K\"ahler metric.

\begin{defn}
    Let $X$ be a compact K\"ahler manifold and $\eta\in H^{1,1}(X,\RR)$ be a pseudoeffective class. We say the class $\eta$ is {\em $p$-modified nef}, if $\nu(\eta,V)=0$, for all varieties $V\subsetneq X$ of dimension $k\geq p$, and a collection of such classes forms a closed cone, which we denote it by  $\sM_p \sN$. We call $\sM_p\sK$ the {\em $p$-modified K\"ahler} cone which is just the interior of $\sM_p \sN$. A class $\eta\in\sM_p\sK$ is called the $p$-modified K\"ahler class.
\end{defn}

We have the inclusions $$\sK=\sM_0\sK\subseteq\sM_1\sK\subseteq\sM_2\sK\subseteq\cdots\subseteq\sM_{n-1}\sK\subseteq\sM_n\sK=\mathfrak{Big}(X)$$
and for any big class $\eta\in\mathfrak{Big}(X)$, the non-K\"ahler locus of $\eta$ denoted by $E_{nK}$ is defined as 
$$ E_{nK}=\bigcap_{T\in \eta}E_+(T), $$
where the intersection is taken over all the positive currents $T$ with analytic singularities.
If $\xi\in\sM_{p+1}\sK$, then the dimension of the connected components of the analytic variety $ E_{nK}(\xi)$ is atmost $p$, which is proved as Lemma 2.3 in \cite{KSD2}.

\subsection{Cycles and slopes}
The cycle-space facts used below are standard, see \cite{BarletCycles,BarletMagnussonII}.
Let \(X\) be a compact Kähler manifold of dimension \(n\), and let
\(\alpha,\beta\in H^{1,1}(X,\mathbb R)\) be Kähler classes. If
\(V\subset X\) is an irreducible analytic subvariety of dimension \(p\),
we write
\[
\int_V \eta
:=
\langle [V],\eta\rangle
\]
for the pairing of its fundamental cycle with a closed \((p,p)\)-class
\(\eta\). Equivalently, if \(\eta\) is represented by a smooth form, this
is integration over \(V_{\mathrm{reg}}\).

We use the slope convention
\[
\mu_{\alpha,\beta}(V)
:=
p\frac{\int_V\beta\wedge\alpha^{p-1}}
{\int_V\alpha^p}.
\]
For \(X\) itself we write
\[
c=\mu_{\alpha,\beta}(X)
=
n\frac{\int_X\beta\wedge\alpha^{n-1}}
{\int_X\alpha^n}.
\]
Thus, $\mu_{\alpha,\beta}(V)=c$
is equivalent to $$
\int_V(c\alpha-p\beta)\wedge\alpha^{p-1}=0.$$

The pluripotential stability threshold is
\[
\Gamma^{pp}_{\beta}(\alpha)
:=
\inf_{V\subsetneq X}
\frac{1}{n-\dim V}
\bigl(\mu_{\alpha,\beta}(X)-\mu_{\alpha,\beta}(V)\bigr),
\]
where the infimum is over proper positive-dimensional irreducible
analytic subvarieties.

A proper $p$-dimensional irreducible subvariety $V\subset X$ is called optimal if \[ \frac{c-\mu_{\alpha,\beta}(V)}{n-p}=\Gamma^{\mathrm{pp}}_\beta(\alpha), \] or equivalently $ \mu_{\alpha,\beta}(V)=c-(n-p)\Gamma^{\mathrm{pp}}_\beta(\alpha). $ 

\begin{defn} Let $X$ be a K\"ahler manifold and $\al$ and $\be$ are two K\"ahler classes.
    \begin{enumerate}[i)]
        \item We call the pair $(\al,\be)$ {\em J-positive ( or stable )},   if $\mu_{\alpha,\beta}(V)< c,$ for every proper \(V\subsetneq X\). In this case, we have $\Gamma^{pp}_{\beta}(\alpha)>0.$
        \item We call the pair $(\al,\be)$ as {\em J-nef},   if $\mu_{\alpha,\beta}(V)\leq c,$ for every proper $V\subsetneq X$.  In this case, we have $\Gamma^{pp}_{\beta}(\alpha)\geq 0.$
        
        We then call a proper subvariety $V$ as {\em optimal} if $\mu_{\alpha,\beta}(V)=c.$
    \item We call the pair $(\al,\be)$ as {\em J-unstable},   if $\mu_{\alpha,\beta}(V) > c,$ for some proper $V\subsetneq X$.  In this case, we have $\Gamma^{pp}_{\beta}(\alpha)<0.$
    \end{enumerate}
\end{defn}

%\subsection{General optimality, shifts, and nef thresholds.} 

\subsubsection{Nef threshold:}
For every $d$-dimensional irreducible subvariety $Z\subset X$ and every $t\in\mathbb R$, one has \[ \mu_{\alpha,\beta+t\alpha}(Z) = \mu_{\alpha,\beta}(Z)+dt, \] and hence $\Gamma^{pp}_{\beta+t\alpha}(\alpha) = \Gamma^{pp}_{\beta}(\alpha)+t.$
In particular, the class \[ \beta_0:=\beta-\Gamma^{\mathrm{pp}}_\beta(\alpha)\alpha \] places the pair numerically at the boundary: $\Gamma^{pp}_{\beta_0}(\alpha)=0,$ and has the same optimal subvarieties as $(\alpha,\beta)$. Here this is a numerical identity; the class $\beta_0$ need not itself be K\"ahler. We shall also use the nef threshold of $\beta$ in the direction of $\alpha$, \[ s_\alpha(\beta) := \sup\{s\in\mathbb R:\ \beta-s\alpha\ \text{is nef}\}. \] 
Since $\alpha$ and $\beta$ are K\"ahler, $s_\alpha(\beta)>0$. Moreover, if $0<\varepsilon<s_\alpha(\beta)$, then $\beta-\varepsilon\alpha$ is K\"ahler. Indeed, one may choose $t>\varepsilon$ such that $\beta-t\alpha$ is nef and write \[ \beta-\varepsilon\alpha = (\beta-t\alpha)+(t-\varepsilon)\alpha. \]

\subsubsection{Moving towards the semistable but not stable case:}\label{subsec:linear}
We now relate the boundary statement to the original unstable pair using \cite[Lemma 14]{ZD20}. 
Consider the affine path
\[
\beta_t
:=
(1-t)\frac{c}{n}\alpha+t\beta,
\qquad
0\leq t\leq1.
\]
Its total slope is independent of $t$:
\[
\mu_{\alpha,\beta_t}(X)=c.
\]

\begin{lem}
\label{lem:optimal-independent-path}
For every proper $p$-dimensional subvariety $V\subset X$ and every
$t>0$,
\[
\frac{c-\mu_{\alpha,\beta_t}(V)}{n-p}
=
(1-t)\frac{c}{n}
+
t\frac{c-\mu_{\alpha,\beta}(V)}{n-p}.
\]
Consequently,
\[
\mathrm{Des}^{\mathrm{opt}}_{\alpha,\beta_t}(X)
=
\mathrm{Des}^{\mathrm{opt}}_{\alpha,\beta}(X)
\qquad
(t>0).
\]
\end{lem}

\begin{proof}
Since $\dim V=p$,
\[
\mu_{\alpha,\beta_t}(V)
=
(1-t)\frac{pc}{n}
+
t\mu_{\alpha,\beta}(V).
\]
The displayed identity follows immediately. The first term on its
right-hand side is independent of $V$, while the second is a positive
multiple of the original destabilizing quotient. Hence the
minimizers are unchanged.
\end{proof}

Also, for the K\"ahler metrics $\chi\in\be$ and $\omega\in\al$ let
\[
\chi_t=(1-t)\frac{c}{n}\omega+t\chi, \qquad
\omega_t=\omega+\sqrt{-1}\,\partial\bar\partial\varphi_t,
\]
 and we consider the following continuity path equation
\begin{equation}\tag{$*_t$}\label{eq:path}
	\Lambda_{\omega_t}\chi_t=c, \qquad \varphi_t\in \mathcal{H}_{\omega}.
\end{equation}
The equation \eqref{eq:path} is solvable at $t=0$, with the solution $\varphi_0=0$. So we consider
\begin{equation}\label{eq:tau}
	\tau:=\sup\{t\in[0,1]:(*_t)\text{ is solvable at }t\}.
\end{equation}
If $\Gamma^{\mathrm{pp}}_{\beta}(\alpha)<0$, then we have $\tau<1$.
Also, by definition we get that
\[
\tau
=
\sup\left\{
t\in[0,1]:
\Gamma^{\mathrm{pp}}_{\beta_t}(\alpha)>0
\right\},
\] and hence $\Gamma^{\mathrm{pp}}_{\beta_\tau}(\alpha)=0.$

Here, the class $\be_\tau$ places the pair numerically at the boundary: $ \Gamma^{pp}_{\beta_\tau}(\alpha)=0,$ and has the same optimal subvarieties as $(\alpha,\beta)$.

\subsection{Families of cycles} 
A compact effective $p$-cycle on $X$ is a finite formal sum \( Z=\sum_i m_iV_i \) where $m_i\in\mathbb Z_{>0}$ and the $V_i$ are distinct irreducible $p$-dimensional analytic subvarieties. Its support is \[ |Z|:=\bigcup_iV_i. \] For a closed $(p,p)$-class $\eta$, we put \[ \int_Z\eta:=\sum_i m_i\int_{V_i}\eta. \] The slope extends to effective $p$-cycles by \[ \mu_{\alpha,\beta}(Z) := p\,\frac{\int_Z\beta\wedge\alpha^{p-1}} {\int_Z\alpha^p}. \] If \[ a_i:=\int_{V_i}\alpha^p>0, \] then \[ \mu_{\alpha,\beta}(Z) = \frac{\sum_i m_i a_i\mu_{\alpha,\beta}(V_i)} {\sum_i m_i a_i}. \] Thus, the slope of a cycle is a positive weighted average of the slopes of its irreducible components. In particular, if \[ \mu_{\alpha,\beta}(Z)=c-(n-p)\Gamma^{\mathrm{pp}}_\beta(\alpha), \] then every irreducible component of $Z$ is an optimal subvariety: the definition of $\Gamma^{\mathrm{pp}}_\beta(\alpha)$ gives \[ \mu_{\alpha,\beta}(V_i)\leq c-(n-p)\Gamma^{\mathrm{pp}}_\beta(\alpha) \] for every $i$, and equality of the weighted average forces equality term by term. Let $\mathcal C_p(X)$ denote the Barlet space of compact $p$-cycles on $X$. An analytic family of $p$-cycles over a reduced complex space $T$ is a holomorphic map \( \tau:T\longrightarrow\mathcal C_p(X)\), \(t\longmapsto Z_t\). The fundamental homology class $[Z_t]$ is locally constant in $t$. Hence, if $T$ is connected, then for every closed $(p,p)$-class $\eta$ the function \[ t\longmapsto\int_{Z_t}\eta \] is constant, and so is $\mu_{\alpha,\beta}(Z_t)$. This constancy holds at every parameter value. The qualification ``for general $t$'' used below refers only to geometric properties such as irreducibility or reducedness of the fiber, which may fail at special parameter values. 

\subsection{Universal families and sweep-outs}
The Barlet space carries a tautological family of cycles. Pulling it back by $\tau:T\to\mathcal C_p(X)$ gives an analytic cycle \[ \mathcal U_T\subset T\times X \] with projections \[ f_0:\mathcal U_T\longrightarrow T, \qquad q_0:\mathcal U_T\longrightarrow X, \] whose cycle-theoretic fiber over $t$ is $Z_t$. Suppose now that $T$ is an irreducible compact curve and that the general cycle $Z_t=V_t$ is irreducible and reduced. Let $\mathcal U_T^{\mathrm{main}}$ be the unique irreducible component of the support of $\mathcal U_T$ which dominates $T$. We define the \emph{sweep-out} of the family by 
\[ W := \bigl(q_0(\mathcal U_T^{\mathrm{main}})\bigr)_{\mathrm{red}} \subset X. \] Equivalently, $W$ is the closure in $X$ of the union of the supports of the general members $V_t$. Since $q_0$ is proper, $W$ is a compact analytic subvariety. If the supports of the general fibers are not all equal, then $\dim W=p+1,$ and \[ q_0:\mathcal U_T^{\mathrm{main}}\longrightarrow W \] is generically finite. Let \[ \nu:Y\longrightarrow\mathcal U_T^{\mathrm{main}} \] be a resolution, and set \[ f:=f_0\circ\nu:Y\longrightarrow T, \qquad g:=q_0\circ\nu:Y\longrightarrow W\subset X. \] Then $Y$ is a smooth compact K\"ahler manifold of dimension $p+1$, the map $g$ is generically finite onto $W$, and a general fiber \[ F_t:=f^{-1}(t) \] maps generically finitely onto $V_t$. Choose a K\"ahler class $\ell$ on $T$ normalized by \[ \int_T\ell=1. \] For \[ A:=g^*\alpha, \qquad C:=g^*\delta, \qquad L:=f^*\ell, \] where $\delta\in H^{1,1}(X,\mathbb R)$, the projection formula gives the following degree-cancellation identities. If $d_W$ is the generic degree of $g$ and $d_t$ that of $F_t\to V_t$, then \[ A^{p+1} = d_W\int_W\alpha^{p+1}, \qquad C\cdot A^p = d_W\int_W\delta\wedge\alpha^p, \] and, for general $t$, \[ L\cdot A^p = d_t\int_{V_t}\alpha^p, \qquad C\cdot L\cdot A^{p-1} = d_t\int_{V_t}\delta\wedge\alpha^{p-1}. \] Consequently, 
\[ \mu_{\alpha,\delta}(W) = (p+1)\frac{C\cdot A^p}{A^{p+1}}, \qquad \mu_{\alpha,\delta}(V_t) = p\frac{C\cdot L\cdot A^{p-1}}{L\cdot A^p}. \] 
The class $A$ is nef and big, since $g$ is generically finite and $A^{p+1}>0$. If $\delta$ is nef, then $C$ is nef. Moreover, $L$ is nef, and we have $L^2=0$ and $L\cdot A^p>0$.

\subsection{The projective case and connected group actions} Suppose that $X$ is projective and fix an embedding $X\hookrightarrow\mathbb P^N$. For fixed dimension $p$ and degree $d$, the Chow variety \( \operatorname{Chow}_{p,d}(X) \) is projective and carries its universal cycle. Let \[ \operatorname{Chow}^{\mathrm{irr}}_{p,d}(X) \subset \operatorname{Chow}_{p,d}(X) \] denote the subset parametrizing irreducible reduced cycles. If an irreducible cycle $[V]$ is not isolated in this subset, algebraic curve selection provides an integral projective curve \( C\subset\operatorname{Chow}_{p,d}(X) \) through $[V]$ whose general point belongs to $\operatorname{Chow}^{\mathrm{irr}}_{p,d}(X)$. After normalizing $C$, the restricted universal cycle gives an analytic family over a smooth compact curve to which the preceding discussion applies. By the functoriality of the Chow scheme for algebraic families of cycles
\cite[Theorem~I.3.21]{Kol96}, an algebraic action
\(G\times X\to X\) induces an algebraic action on the corresponding
Chow spaces by
\[
(g,[Z])\longmapsto [gZ].
\] Since every element of $G$ is homotopic to the identity, $G$ acts trivially on cohomology. It therefore preserves $\alpha$, $\beta$, and the slope $\mu_{\alpha,\beta}$. For an irreducible cycle $V$, the orbit $G\cdot[V]$ consists entirely of irreducible cycles of the same numerical class. In particular, if this orbit is positive-dimensional, then $[V]$ is not isolated among irreducible cycles.

\section{Finitely many optimally destabilizing curves}
    
\noindent In this section we first prove finiteness of optimally destabilizing \emph{curves} on $X$ of arbitrary dimension. We do this by observing that the class $c\al-\be$ is always big on $X$ if the pair $(\al,\be)$ is J-nef. On complex surfaces this follows easily, as we can reduce the J-equation to the complex Monge-Amp\'ere equation which is also proved in \cite{KSD1}, but in the higher dimensions we need the following lemma.
    
\begin{lem}\label{lem:positive volume}
    Let $X$ be a compact K\"ahler manifold of complex dimension $n\ge 2$, and let $\alpha,\beta$ be K\"ahler classes. Suppose that $\mu_{\al,\be}(V)\leq c$ for all proper irreducible subvarieties $V\subset X$, where 
\[\label{eq:c}
	c:=n\frac{\beta\cdot\alpha^{n-1}}{\alpha^n}.
\]
Then we have
\[\label{eq:positive-top}
	\int_X(c\alpha-\beta)^n>0.
\]
\end{lem}

\begin{proof}

Consider the following shorthand notation: For any $0\le k\le n$ write	$S_k:=\alpha^{n-k}\cdot\beta^k $
and for $1\le k\le n$ define
\[	r_k:=\frac{S_k}{S_{k-1}}.\]

By the Khovanskii--Teissier inequalities for nef classes $\alpha,\beta$, the sequence $k\mapsto\log(S_k)$ is, concave or equivalently $S_k^2\geq S_{k-1}S_{k+1}$. Hence, the sequence $(r_k)_{k=1}^n$ is non-increasing, and in particular we deduce that
\begin{equation}\label{eq:r-bound}
	r_{k+1}\le r_1=\frac{S_1}{S_0}=\frac{\beta\cdot\alpha^{n-1}}{\alpha^n}=\frac{c}{n}
\end{equation}
holds, for all $0\le k\le n$.

We now expand and regroup the binomial terms in consecutive pairs as follows
\begin{align*}
	\int_X(c\alpha-\beta)^n
	&=\sum_{k=0}^n(-1)^k\binom{n}{k}c^{n-k}S_k \\
	&=\sum_{j=0}^{\lfloor(n-1)/2\rfloor}\left(\binom{n}{2j}c^{n-2j}S_{2j}-\binom{n}{2j+1}c^{n-2j-1}S_{2j+1}\right)
	+\mathbf{1}_{\{n\text{ even}\}}S_n. 
\end{align*}
where $\mathbf{1}_{\{n\text{ even}\}}$ is the indicator function. Each term in brackets factors as follows,
\[
	\binom{n}{2j}c^{n-2j-1}S_{2j}\left(c-\frac{n-2j}{2j+1}\frac{S_{2j+1}}{S_{2j}}\right) =\binom{n}{2j}c^{n-2j-1}S_{2j}\left(c-\frac{n-2j}{2j+1}r_{2j+1}\right).	
\]
Imposing \eqref{eq:r-bound} we deduce from this a uniform lower bound
\[
	c-\frac{n-2j}{2j+1}r_{2j+1}
	\ge c-\frac{n-2j}{2j+1}\frac{c}{n}
	=\frac{c}{n}\cdot\frac{2j(n+1)}{2j+1}.
\]
Hence for every $j\ge 1$, we have
\[
	\binom{n}{2j}c^{n-2j}S_{2j}-\binom{n}{2j+1}c^{n-2j-1}S_{2j+1}
	\ge \binom{n}{2j}c^{n-2j-1}S_{2j}\cdot\frac{c}{n}\cdot\frac{2j(n+1)}{2j+1}>0.
\]
In the case $j=0$, the expression in the brackets is non-negative. Finally, if $n$ is even, there is also the leftover term $S_n=\beta^n$, which is positive since $\beta$ is assumed to be a K\"ahler class. Summing these contributions, we conclude that
\[
	\int_X(c\alpha-\beta)^n>0,
\]
finishing the proof.
\end{proof}

\begin{cor}
    Let $\alpha$ and $\beta$ be two K\"ahler classes on a K\"ahler manifold $X$ with $\dim X>2$ such that $(\alpha,\beta)$ is J-nef. Then $c\alpha-\beta$ is big and modified K\"ahler. Moreover, $c\alpha-\beta\in \mathcal{M}_2\mathcal{K}$.
\end{cor}

\begin{proof}
   Given that $(\al,\be)$ is J-nef, this implies that $\Gamma^{pp}_{\be_t}(\al)>0,$ for all $t\in[0,1).$ Therefore, by Corollary 1.2 of \cite{So20}, $c\al-\be_t$ is K\"ahler, for all $t\in[0,1)$, so $c\al-\be$ is nef.
    Then by Lemma \ref{lem:positive volume} and \cite[Theorem 2.12]{DP04}, we see that $c\alpha-\beta$ is big.

To prove the modified K\"ahler condition, let $V\subset X$ be any variety with $2\leq\dim V\leq n-1$. Then a calculation similar to Lemma \ref{lem:positive volume} using Khovanskii-Teissier inequality on $V$ implies that
\[
\int_V(c\alpha-\beta)^{\dim V}>0.
\]
Then by \cite[Lemma 2.3]{KSD2}, since for nef and big classes the non-Kähler locus equals the null locus, we have $c\alpha-\beta\in \mathcal{M}_2\mathcal{K}$. Thus, $c\alpha-\beta$ is a $2$-modified K\"ahler class.
\end{proof}

\noindent As a consequence, we have the following corollary, which we stated at the beginning of this section.

\begin{cor}\label{cor:finite curves}
	Let \(X\) be a compact Kähler manifold of dimension \(n>2\), and let \(\alpha,\beta\) be Kähler classes such that \((\alpha,\beta)\) is J-nef. Then the non-Kähler locus
$E_{\rm nK}(c\alpha-\beta) $
has no components of dimension two or higher and contains only finitely many curves. These curves are precisely the optimally destabilizing curves of \((X,\alpha,\beta)\).
%Assume that $(\al,\be)$ is $J$-nef. Then the non-K\"ahler locus $E_{nK}(c\alpha - \beta)$ contains only finitely many curves, and no higher-dimensional components, and these curves are precisely the optimally destabilizing curves of the pair $(\al,\be)$.
\end{cor}
\medskip
%\noindent\textbf{Remark 6.} \emph{The above corollary implies that the non-K\"ahler locus of the class $c\alpha-\beta$ denoted by $E_{\mathrm{nK}}(c\alpha-\beta)$ contains only finitely many curves.}

\medskip 

\section{A fixed analytic set containing the optimal destabilizing locus}
\label{sec:optimal-analytic-envelope}

\noindent With the goal of proving finiteness of optimally destabilizing divisors in arbitrary dimension, we show that the optimal destabilizing subvarieties are contained in a proper analytic set of $X$, by drawing closely on the techniques of G. Chen \cite{Che21}. 

Let $X$ be a compact K\"ahler manifold of dimension $n$, and let
$\alpha,\beta\in H^{1,1}(X,\mathbb R)$ be K\"ahler classes. Fix
K\"ahler representatives
\(
\omega \in \alpha,
\chi\in\beta,
\)
and write
\[
c:=\mu_{\alpha,\beta}(X)
=
n\frac{\beta\cdot\alpha^{n-1}}{\alpha^n}.
\]

For a smooth positive $(1,1)$-form $A$, set
\[
P_\chi(A)(x)
:=
\max_{\substack{H\subset T_x^{1,0}X\\ \dim_{\mathbb C}H=n-1}}
\operatorname{tr}_{A|_H}(\chi|_H).
\]
Thus
\[
P_\chi(A)<c
\]
is equivalent to the strict cone condition
\[
cA^{n-1}-(n-1)\chi\wedge A^{n-2}>0.
\]
For positive currents, we use the local-convolution interpretation of
this inequality introduced in \cite[Definition~3.3]{Che21}.
For an arbitrary pair $(\alpha,\beta)$, define
\[
\mathrm{Des}^{\mathrm{opt}}_{\alpha,\beta}(X)
:=
\left\{
V\subsetneq X:
\frac{c-\mu_{\alpha,\beta}(V)}
     {n-\dim V}
=
\Gamma^{\mathrm{pp}}_\beta(\alpha)
\right\}.
\]
We call
\[
\mathcal D^{\mathrm{opt}}_{\alpha,\beta}
:=
\bigcup_{V\in
\mathrm{Des}^{\mathrm{opt}}_{\alpha,\beta}(X)}
V
\]
the optimal destabilizing locus. When
$\Gamma^{\mathrm{pp}}_\beta(\alpha)=0,$
one has
\[
\mathrm{Des}^{\mathrm{opt}}_{\alpha,\beta}(X)
=
\{V\subsetneq X:\mu_{\alpha,\beta}(V)=c\}.
\]

\subsection{The mass-concentration current}

We first record the two consequences of G. Chen's \cite{Che21} argument, which will be
used below.

\begin{lem}
\label{lem:boundary-mass-concentration}
Assume that
${\Gamma^{\mathrm{pp}}_\beta(\alpha) = 0}.$
Then the following hold.

\begin{enumerate}
\item For every $s>0$, there exists a smooth K\"ahler form
$\Omega_s\in(1+s)\alpha$
such that
\[
P_\chi(\Omega_s)<c.
\]

\item There exist $\varepsilon_0>0$ and a closed positive current
$S\in\alpha-\varepsilon_0\beta$ 
such that
$P_\chi(S)\leq c$
in the local-convolution sense.
Consequently, the current
$T_{\mathrm{mc}}:=S+\varepsilon_0\chi\in\alpha$ 
satisfies
\[
    T_{\mathrm{mc}}\geq\varepsilon_0\chi \qquad \text{ and }\qquad
P_\chi(T_{\mathrm{mc}})
\leq
c_{\varepsilon_0}
:=
\frac{c}
{1+\frac{\varepsilon_0c}{n-1}}
<c
\]

in the local-convolution sense.
\end{enumerate}
\end{lem}

\begin{proof}
We first explain the strict approximation. For $s>0$, put
\[
\delta_s
:=
\frac{cs}{2(n-1)(1+s)}.
\]
Let $V\subsetneq X$ be an irreducible $p$-dimensional subvariety.
Since $\Gamma^{\mathrm{pp}}_\beta(\alpha)=0$, one has
\[
\int_V
\left(
c\alpha^p-p\beta\wedge\alpha^{p-1}
\right)
\geq0.
\]
A direct calculation gives
\[
\begin{aligned}
&
\int_V
\left[
\bigl(c-(n-p)\delta_s\bigr)
\bigl((1+s)\alpha\bigr)^p
-
p\beta\wedge
\bigl((1+s)\alpha\bigr)^{p-1}
\right]
\\
&\quad=
(1+s)^{p-1}
\int_V
\left[
c\alpha^p-p\beta\wedge\alpha^{p-1}
+
\bigl(cs-(n-p)\delta_s(1+s)\bigr)\alpha^p
\right].
\end{aligned}
\]
Since $n-p\leq n-1$,
\[
cs-(n-p)\delta_s(1+s)
\geq\frac{cs}{2}>0.
\]
%Thus Chen's uniform numerical inequalities hold for the class
%$(1+s)\alpha$. Applying the solvability theorem of
%\cite{Che21}, with the corresponding positive top-degree term,
%produces $\Omega_s\in(1+s)\alpha$ satisfying
%\[
%P_\chi(\Omega_s)<c.
%\]

Thus all the hypotheses of \cite[Theorem~1.11]{Che21} are satisfied for the class $(1+s)\alpha$. More explicitly, put \[ f_s := \frac{cs(1+s)^{n-1}\alpha^n}{\beta^n}>0. \] The cohomological compatibility condition is \[ \begin{aligned} f_s\beta^n &= c(1+s)^n\alpha^n - n(1+s)^{n-1}\beta\cdot\alpha^{n-1} \\ &= cs(1+s)^{n-1}\alpha^n, \end{aligned} \]
where we used $ n\beta\cdot\alpha^{n-1}=c\alpha^n. $ Consequently, \cite[Theorem~1.11]{Che21} produces a K\"ahler form \[ \Omega_s\in(1+s)\alpha \] satisfying \[ n\chi\wedge\Omega_s^{n-1} + f_s\chi^n = c\Omega_s^n \] and \[ P_\chi(\Omega_s)<c. \] The family $\{\Omega_s\}_{s>0}$ therefore satisfies the hypotheses of the mass-concentration theorem \cite[Theorem~1.18]{Che21}. Hence there exist $\varepsilon_0>0$ and a closed positive current $ S\in\alpha-\varepsilon_0\beta $ such that \( P_\chi(S)\leq c \) in the local-convolution sense.
The mass-concentration theorem
\cite[Theorem~1.18]{Che21} now gives
$\varepsilon_0>0$ and a positive current
\(
S\in\alpha-\varepsilon_0\beta
\)
satisfying $P_\chi(S)\leq c$ in the weak sense.

It remains to verify the strict inequality for
$T_{\mathrm{mc}}=S+\varepsilon_0\chi$. Work on a coordinate chart
and let $\chi_0\leq\chi$ be a constant-coefficient K\"ahler form.
If $S_\rho$ is a local convolution regularization, let
$\lambda_1,\ldots,\lambda_n$ be the eigenvalues of $\chi_0$
relative to $S_\rho$. The weak cone inequality gives
\[
\sum_{j\neq k}\lambda_j\leq c
\qquad
(1\leq k\leq n).
\]
After adding $\varepsilon_0\chi_0$, the corresponding eigenvalues
become
$\frac{\lambda_j}{1+\varepsilon_0\lambda_j}.$
The function
\[
x\longmapsto\frac{x}{1+\varepsilon_0x}
\]
is increasing and concave. Hence Jensen's inequality gives
\[
\begin{aligned}
\sum_{j\neq k}
\frac{\lambda_j}{1+\varepsilon_0\lambda_j}
&\leq
(n-1)
\frac{
\frac{1}{n-1}\sum_{j\neq k}\lambda_j
}{
1+\frac{\varepsilon_0}{n-1}
\sum_{j\neq k}\lambda_j
}
\\
&\leq
\frac{c}
{1+\frac{\varepsilon_0c}{n-1}}
=
c_{\varepsilon_0}.
\end{aligned}
\]
Since convolution preserves $\chi-\chi_0\geq0$, and since
$P_{\chi_0}$ is order-reversing, the same bound holds for the
regularizations of $T_{\mathrm{mc}}$. This proves the lemma.
\end{proof}

For $\delta>0$, define the Siu upper-level set
\[
E_\delta(T_{\mathrm{mc}})
:=
\{x\in X:\nu(T_{\mathrm{mc}},x)\geq\delta\}.
\]
By Siu's theorem, this is a proper analytic subset of $X$.

The next lemma is the key point of the argument. It is a partial
version of Chen's regularization \cite[Section~4]{Che21} construction: we do not attempt to
regularize $T_{\mathrm{mc}}$ globally while preserving a uniform
strict cone margin. We only require such a margin on compact subsets
outside one fixed Siu level.

\begin{lem}
\label{lem:partial-cone-regularization}
There exists $\delta_0>0$ with the following property. Let
$K\Subset X\setminus E_{\delta_0}(T_{\mathrm{mc}}).$
Then there are constants $a_K>0,\eta_K>0,$ and $ s_K>0,$ 
such that, for every $0<s<s_K$, there exists a smooth K\"ahler form
$\widehat\Omega_{s,K}\in(1+s)\alpha$
satisfying
\[
P_\chi(\widehat\Omega_{s,K})<c
\qquad\text{on }X,
\]
and
\[
\widehat\Omega_{s,K}\geq a_K\chi,
\qquad
P_\chi(\widehat\Omega_{s,K})\leq c-\eta_K
\qquad\text{on }K.
\]
\end{lem}

\begin{proof} 
Set \[ c_0 := \frac{c} {1+\frac{\varepsilon_0c}{n-1}} <c, \] so that Lemma~\ref{lem:boundary-mass-concentration} gives 
\[ T_{\mathrm{mc}}\geq\varepsilon_0\chi, \qquad P_\chi(T_{\mathrm{mc}})\leq c_0 \] 
in the local-convolution sense. We first introduce the shrinking step which is used in \cite[Section~4]{Che21}. 
Choose a number $\sigma>0$ sufficiently small that \[ \lambda_0:=1-100\sigma>0 \] and \begin{equation}\label{ineq: c}
     \frac{(1+\sigma)c_0}{\lambda_0}<c.
\end{equation}
Fix also $s_0\in(0,1)$. Choose a finite collection of coordinate balls 
\[ B_r(x_i)\Subset B_{2r}(x_i), \qquad 1\leq i\leq N, \] such that the balls $B_r(x_i)$ cover $X$. After decreasing $r$ and $\sigma$, choose on $B_{2r}(x_i)$: 
\begin{enumerate} 
    \item a local K\"ahler potential $\psi_i$ for $\omega$, i.e.,$\sqrt{-1}\partial\bar\partial\psi_i=\omega; $
    \item a constant-coefficient K\"ahler form $\chi_i^0$ satisfying $ \chi_i^0\leq\chi\leq(1+\sigma)\chi_i^0; $
    \item the normalizations and coordinate estimates used in the setup preceding \cite[Proposition~4.1]{Che21}, with Chen's constant $\epsilon_{4.1}$ replaced by $\sigma$. 
\end{enumerate}

Write 
\[ T_{\mathrm{mc}} = \omega+\sqrt{-1}\partial\bar\partial\varphi \]
for a global quasi-plurisubharmonic function $\varphi$. For $0<s<s_0$, define the shrunk auxiliary current 
\[\widetilde R_s := \lambda_0\bigl(T_{\mathrm{mc}}+s\omega\bigr). \] Then 
$ [\widetilde R_s] = \lambda_0(1+s)\alpha $ and 
\begin{equation}\label{eq:lelong numbers}
    \nu(\widetilde R_s,x) = \lambda_0\nu(T_{\mathrm{mc}},x) \qquad (x\in X).
\end{equation} 
In particular, its Lelong numbers are independent of $s$ up to the fixed factor $\lambda_0$. A local plurisubharmonic potential for $\widetilde R_s$ on $B_{2r}(x_i)$ is 
\[ \Phi_{s,i} := \lambda_0\bigl(\varphi+(1+s)\psi_i\bigr).\]
Let $\Phi_{s,i,\rho}$ denote the corrected convolution regularization used in \cite[Section~4]{Che21}, at scale $0<\rho<r/20$, and define the corresponding local plurisubharmonic potential relative to the background $(1+s)\omega$ by
\begin{equation}\label{fun:local branch rel}
    v_{s,i,\rho} := \Phi_{s,i,\rho}-(1+s)\psi_i.
\end{equation} 
Thus 
\[ (1+s)\omega + \sqrt{-1}\partial\bar\partial v_{s,i,\rho} = \sqrt{-1}\partial\bar\partial\Phi_{s,i,\rho}. \] We next prove the uniform positivity and cone estimates for these potentials. Since adding the positive form $s\omega$ can only improve the cone inequality, the local-convolution definition gives 
\[ P_{\chi_i^0} \bigl((T_{\mathrm{mc}}+s\omega)_\rho\bigr) \leq c_0. \] By homogeneity of $P_{\chi_i^0}$, \[ P_{\chi_i^0} \bigl(\sqrt{-1}\partial\bar\partial\Phi_{s,i,\rho}\bigr) \leq \frac{c_0}{\lambda_0}. \] Since $\chi\leq(1+\sigma)\chi_i^0$, it follows from \eqref{ineq: c} that \[ P_\chi \bigl(\sqrt{-1}\partial\bar\partial\Phi_{s,i,\rho}\bigr) \leq \frac{(1+\sigma)c_0}{\lambda_0} =:c_1<c.\] Likewise, $ T_{\mathrm{mc}}+s\omega \geq \varepsilon_0\chi \geq \varepsilon_0\chi_i^0, $ so, after convolution, \[ \sqrt{-1}\partial\bar\partial\Phi_{s,i,\rho} \geq \lambda_0\varepsilon_0\chi_i^0 \geq \frac{\lambda_0\varepsilon_0}{1+\sigma}\chi. \] 

Thus there are fixed constants 
$ a_0 := \frac{\lambda_0\varepsilon_0}{1+\sigma}>0, $ and $ \eta_0:=c-c_1>0, $
such that every local potential satisfies 
\begin{equation}\label{ineq:local bbranch control}
    \sqrt{-1}\partial\bar\partial\Phi_{s,i,\rho} \geq a_0\chi, \qquad P_\chi \bigl(\sqrt{-1}\partial\bar\partial\Phi_{s,i,\rho}\bigr) \leq c-\eta_0.
\end{equation}  
The constants in \eqref{ineq:local bbranch control} are independent of $s$ and $\rho$. The role of the factor $\lambda_0=1-100\sigma$ is visible directly in \eqref{fun:local branch rel}. The local potential contains the background contribution \[ -100\sigma(1+s)\psi_i, \] up to the harmless convolution error. In the normalized coordinates, this is the negative quadratic term used in \cite[Proposition~4.1(3)]{Che21} to force a potential defined on an interior coordinate ball to dominate potentials approaching the outer boundary of another chart. For completeness, we now match the remaining notation with \cite[Proposition~4.1]{Che21}.

Let 
\[ \widehat\Phi_{s,i,\rho}(x) := \sup_{B_\rho^i(x)}\Phi_{s,i}, \] and define the approximate Lelong number 
\[ \nu_{s,i}(x,\rho) := \frac{ \widehat\Phi_{s,i,r/16}(x) - \widehat\Phi_{s,i,\rho}(x) }{ \log(r/16)-\log\rho }.\] 
Let $\varepsilon_*>0$ be the number obtained from $\sigma$, $r$, and the fixed convolution kernel by the formula defining $\epsilon_{4.5}$ in the proof of \cite[Proposition~4.1]{Che21}. Set \[ \delta_0:=\frac{\varepsilon_*}{\lambda_0} \] and 
\[ Z:=E_{\delta_0}(T_{\mathrm{mc}}). \] By \eqref{eq:lelong numbers}, \[ Z = \{x\in X:\nu(\widetilde R_s,x)\geq\varepsilon_*\}, \]
and this set is independent of $s$.
Let now \[ K\Subset X\setminus Z. \] 
Choose open neighborhoods $ Z\subset O'\Subset O $ such that 
\begin{equation}\label{eq:set equal}
    \overline O\cap K=\varnothing. 
\end{equation} 
Because $\overline{X\setminus O'}$ is compact and disjoint from $Z$, upper semicontinuity of Lelong numbers gives
\begin{equation}\label{ineq:lelong cont}
    \sup_{x\in X\setminus O'} \lambda_0\nu(T_{\mathrm{mc}},x) < \varepsilon_*.
\end{equation}  
Let $\Omega_s = (1+s)\omega + \sqrt{-1}\partial\bar\partial u_s $ be the smooth strict-cone form given by the first part of Lemma~\ref{lem:boundary-mass-concentration}. We use on $O$ the shifted smooth comparison potential
\begin{equation}\label{fun:comparison branch}
    u_s+3\varepsilon_*\log\rho.
\end{equation}  Notice that $u_s$ is smooth on all of $X$. This is the simplification which removes the need to assume that $Z$ is smooth or to resolve its singularities. The proof of \cite[Proposition~4.1]{Che21} now applies to the potentials $\Phi_{s,i}$ and the comparison potential \eqref{fun:comparison branch}, with Chen's parameters \[ \epsilon_{4.1}=\sigma, \qquad \epsilon_{4.5}=\varepsilon_*. \] The added term $\lambda_0s\psi_i$ varies in a uniformly bounded family of smooth functions for $0<s<s_0$, so all local constants in that proof can be chosen independently of $s$. The bounds of the smooth function $u_s$ may depend on $s$; this only affects how small the regularization scale $\rho$ must be for that particular $s$. More precisely, after choosing $\rho=\rho(s,K)>0$ sufficiently small, the three conclusions of \cite[Proposition~4.1]{Che21} give: 
\begin{enumerate}
    \item on a neighborhood of $\partial O$, the local convolution potentials dominate the shifted comparison potential \eqref{fun:comparison branch}; 
    \item if a local convolution potential can be compared with the comparison potential on $O$, then all relevant approximate Lelong numbers are strictly smaller than $4\varepsilon_*$; 
    \item whenever the relevant approximate Lelong numbers are at most $4\varepsilon_*$, a potential approaching the outer boundary of one coordinate chart is dominated, with a fixed positive gap, by a potential defined on an interior chart.
\end{enumerate}
    The first assertion follows from \eqref{ineq:lelong cont} and \cite[Proposition~4.1(1)]{Che21}; the second and third are precisely the roles of parts (2) and (3) of that proposition. These three separation properties ensure that the finite regularized maximum of \[ u_s+3\varepsilon_*\log\rho \quad\text{on }O \] and \[ v_{s,i,\rho} \quad\text{on the coordinate balls} \] defines one globally smooth function, which we denote by $\widehat u_{s,K}$. Set 
    \begin{equation}\label{eqn:omega widehat}
        \widehat\Omega_{s,K} := (1+s)\omega + \sqrt{-1}\partial\bar\partial\widehat u_{s,K}.
    \end{equation} 
    Since $\widehat u_{s,K}$ is globally defined, $ [\widehat\Omega_{s,K}]=(1+s)\alpha. $ 
    At a transition point, the complex Hessian of a regularized maximum has the form 
    \begin{equation}\label{eqns:reg max}
        \sum_j\lambda_jA_j+Q, \qquad \lambda_j\geq0, \qquad \sum_j\lambda_j=1, \qquad Q\geq0,
    \end{equation}  
    where the $A_j$ are the forms associated with the active potentials. Every local potential satisfies \eqref{ineq:local bbranch control}, while the comparison potential is the K\"ahler form $\Omega_s$ and satisfies \[ P_\chi(\Omega_s)<c. \] 
Convexity and order reversal of $P_\chi$ therefore give 
\[ P_\chi(\widehat\Omega_{s,K})<c \] on all of $X$. In particular, $\widehat\Omega_{s,K}$ is K\"ahler. Finally, by \eqref{eq:set equal} the compact set $K$ lies outside the region where the comparison branch is needed. By \eqref{ineq:lelong cont}, after decreasing $\rho(s,K)$, the approximate Lelong numbers on $K$ are at most $2\varepsilon_*$. 
Part (1) of \cite[Proposition~4.1]{Che21} then shows that the comparison potential is inactive on $K$. Thus only the local potentials satisfying \eqref{ineq:local bbranch control} participate in the regularized maximum on $K$. Equations \eqref{ineq:local bbranch control} and \eqref{eqns:reg max} consequently imply \[ \widehat\Omega_{s,K}\geq a_0\chi, \qquad P_\chi(\widehat\Omega_{s,K})\leq c-\eta_0 \] on $K$. 
Taking $a_K:=a_0, \eta_K:=\eta_0,$ and $ s_K:=s_0 $ proves the lemma. 
\end{proof}

\subsection{Analytic containment}

In this section, we prove that the Siu's upper-level set contains the optimal destabilizing subvarieties in the case $ \Gamma^{pp}_\be(\al)\leq 0 $.

\begin{prop}
\label{prop:optimal-contained-siu}
Assume that $ \Gamma^{pp}_\be(\al)\leq 0 $.
Then there exist a K\"ahler current $ T_{\mathrm{mc}}\in\alpha $ and a number $\delta_0>0$ such that
\[
\mathcal D^{\mathrm{opt}}_{\alpha,\beta}
\subset
E_{\delta_0}(T_{\mathrm{mc}}).
\]
In particular, the optimal destabilizing locus is contained in one
fixed proper analytic subset of $X$.
\end{prop}

\begin{proof}
By Lemma~\ref{lem:optimal-independent-path}, we have $\mathrm{Des}^{\mathrm{opt}}_{\alpha,\beta_\tau}(X)
=\mathrm{Des}^{\mathrm{opt}}_{\alpha,\beta}(X)$, so without loss of generality we can assume that $ \Gamma^{pp}_\be(\al)= 0 $.

Let $T_{\mathrm{mc}}$ and $\delta_0$ be given by
Lemmas~\ref{lem:boundary-mass-concentration} and
\ref{lem:partial-cone-regularization}. Put
\[
Z:=E_{\delta_0}(T_{\mathrm{mc}}).
\]

Suppose that an irreducible $p$-dimensional optimal destabilizing
subvariety $V$ is not contained in $Z$. Choose
$x\in V_{\mathrm{reg}}\setminus Z $ 
and a relatively compact open set
$ U\Subset X\setminus Z  $ such that
\[
U\cap V_{\mathrm{reg}}\neq\varnothing.
\]
Choose also a compact set
$ K\Subset X\setminus Z $ whose interior contains $\overline U$.

For sufficiently small $s>0$, let
\[
\widehat\Omega_s
:=
\widehat\Omega_{s,K}\in(1+s)\alpha
\]
be the form supplied by
Lemma~\ref{lem:partial-cone-regularization}. On $K$ one has
\begin{equation}\label{eq:omega low bound}
    \widehat\Omega_s\geq a_K\chi,
\qquad
P_\chi(\widehat\Omega_s)\leq c-\eta_K.
\end{equation}

At a smooth point of $V$, let $W=T_x^{1,0}V$. The elementary trace
identity gives
\[
p\chi\wedge\widehat\Omega_s^{p-1}\big|_W
=
\operatorname{tr}_{\widehat\Omega_s|_W}(\chi|_W)
\widehat\Omega_s^p\big|_W.
\]
Since the trace on a $p$-plane is bounded above by
$P_\chi(\widehat\Omega_s)$, equation \eqref{eq:omega low bound} yields
\[
\begin{aligned}
c\widehat\Omega_s^p
-
p\chi\wedge\widehat\Omega_s^{p-1}
&\geq
\eta_K\widehat\Omega_s^p
\\
&\geq
\eta_Ka_K^p\chi^p
\end{aligned}
\]
on $U\cap V_{\mathrm{reg}}$.
\iffalse
It follows that
\[
\int_{U\cap V_{\mathrm{reg}}}
\left(
c\widehat\Omega_s^p
-
p\chi\wedge\widehat\Omega_s^{p-1}
\right)
\geq
\delta_V,
\tag{5.3}
\]
where
\[
\delta_V
:=
\eta_Ka_K^p
\int_{U\cap V_{\mathrm{reg}}}\chi^p
>0
\]
is independent of $s$.
\fi
Globally,
$P_\chi(\widehat\Omega_s)<c,$
so the form
\[
c\widehat\Omega_s^p
-
p\chi\wedge\widehat\Omega_s^{p-1}
\]
is nonnegative on all of $V_{\mathrm{reg}}$. Therefore
\begin{equation}
    \label{ineq:local mass} \int_V
\left(
c\widehat\Omega_s^p - p\chi\wedge\widehat\Omega_s^{p-1} \right)
\geq \int_{U\cap V_{\mathrm{reg}}}
\left(
c\widehat\Omega_s^p
-
p\chi\wedge\widehat\Omega_s^{p-1}
\right) \geq 
\delta_V,
\end{equation}
where
\[
\delta_V
:=
\eta_Ka_K^p
\int_{U\cap V_{\mathrm{reg}}}\chi^p
>0
\]
is independent of $s$.
On the other hand, since $
[\widehat\Omega_s]=(1+s)\alpha,$ 
one has
\[
\begin{aligned}
\int_V
\left(
c\widehat\Omega_s^p
-
p\chi\wedge\widehat\Omega_s^{p-1}
\right)
&=
(1+s)^{p-1}
\int_V
\left(
c(1+s)\alpha^p
-
p\beta\wedge\alpha^{p-1}
\right).
\end{aligned}
\]
Optimality at the boundary gives
\[
\int_V
\left(
c\alpha^p-p\beta\wedge\alpha^{p-1}
\right)=0.
\]
Consequently,
\[
\int_V
\left(
c\widehat\Omega_s^p
-
p\chi\wedge\widehat\Omega_s^{p-1}
\right)
=
cs(1+s)^{p-1}\int_V\alpha^p
\longrightarrow0
\]
as $s\to0$. This contradicts the inequality \eqref{ineq:local mass}. Hence
\[
V\subset Z.
\]
Since $V$ was arbitrary, the proposition follows.
\end{proof}

\begin{cor}
\label{cor:finite-optimal-divisors}
Assume that $ \Gamma^{pp}_\be(\al)\leq 0 $. Then there are only finitely many optimal destabilizing prime divisors.
\end{cor}

\begin{proof}
 By Proposition~\ref{prop:optimal-contained-siu}, every optimal
destabilizing prime divisor is contained in the fixed proper analytic
set
\(
E_{\delta_0}(T_{\mathrm{mc}}).
\)
An irreducible divisor contained in this set is necessarily one of
its codimension-one irreducible components. A compact analytic set has
only finitely many irreducible components, proving the assertion.
\end{proof}

%Combining the two endpoint results with the argument of
%\cite{KSD2} gives the desired conditional finiteness theorem.

\noindent We prove the finiteness of optimal destabilizing varieties by assuming a slightly weaker condition than \cite[Theorem 3.9]{KSD2}. As we already established the finiteness of curves and divisors, we can use these results to drop the assumptions on the modified K\"ahler condition for the classes $ c\al-\be $ and $ c\al-(n-1)\be $.

\begin{cor}
\label{cor:conditional-full-finiteness}
Assume $ \Gamma^{pp}_\be(\al)\leq 0 $ and 
\[
c\alpha-p\beta\in\mathcal M_{p+1}\mathcal K
\]
for every $2\leq p\leq n-2$.
Then $ \mathrm{Des}^{\mathrm{opt}}_{\alpha,\beta}(X) $ is finite.
\end{cor}

\begin{proof}
The proof is the same as in \cite{KSD2}, using that in addition the optimal destabilizing curves are finite by the nef-and-big
argument of Corollary \ref{cor:finite curves}, and the optimal
destabilizing divisors are finite by
Corollary \ref{cor:finite-optimal-divisors}.
Indeed, given this, let $2\leq p\leq n-2$, and let $V$ be a $p$-dimensional optimal
destabilizing subvariety. Then
\(
V\subset
E_{\mathrm{nK}}(c\alpha-p\beta)
\)
by the argument of \cite[Lemma 2.4]{KSD2}. Since
\(
c\alpha-p\beta\in\mathcal M_{p+1}\mathcal K,
\)
its non-K\"ahler locus has dimension at most $p$. Hence every
$p$-dimensional optimal destabilizing subvariety is an irreducible
component of this fixed analytic set, $E_{nK}(c\al-p\be)$. There are only finitely many
such components. This proves the assertion.
\end{proof}

\begin{rem} \label{rem:comparison-liu-analytic-envelope} The final mass comparison in Proposition~\ref{prop:optimal-contained-siu} is closely related to the forced-locus argument in \cite{Liu1,Liu2}. The analytic inputs are nevertheless different. Liu begins with a weak cone current of the form \[ T=S+\varepsilon\chi, \] where $S$ already has analytic singularities, and proves that every numerically null subvariety is contained in the positive-Lelong locus $E_+(T)$. In higher dimensions, the availability of such a current is the additional J-bigness hypothesis. Here we begin instead with the mass-concentration current $T_{\mathrm{mc}}$ supplied by \cite[Theorem~1.18]{Che21}; its singularities need not be analytic. Lemma \ref{lem:partial-cone-regularization} regularizes this current away from one fixed Siu upper-level set \[ E_{\delta_0}(T_{\mathrm{mc}}), \] and the same type of local-mass contradiction then gives \[ \mathcal D^{\mathrm{opt}}_{\alpha,\beta} \subset E_{\delta_0}(T_{\mathrm{mc}}). \]

\iffalse
Our conclusion is weaker than Liu's analytic identification: $E_{\delta_0}(T_{\mathrm{mc}})$ may contain extraneous components, and we do not prove that $T_{\mathrm{mc}}$ has analytic singularities or assume that the pair is J-big. The gain is that the containment theorem requires neither analytic singularities nor a J-bigness assumption. 
\fi

Proposition~\ref{prop:optimal-contained-siu} instead starts from the
possibly non-analytic mass-concentration current of \cite{Che21}.
The partial regularization lemma replaces the full polar set by one
fixed Siu upper-level set
\[
E_{\delta_0}(T_{\mathrm{mc}}).
\]
The conclusion is correspondingly weaker: we prove containment in a
possibly larger analytic set, rather than equality with an analytic
non-ample locus. The advantage is that no analytic-singularity or
J-bigness assumption is imposed.
\end{rem}

\iffalse
\begin{cor}
\label{cor:original-optimal-analytic}
Assume that $\Gamma^{pp}_\be(\al)\leq 0.$ Then there exist a K\"ahler current $T_{\mathrm{mc}}\in\alpha$ and a
number $\delta_0>0$ such that
\[
\mathcal D^{\mathrm{opt}}_{\alpha,\beta}
=
\mathcal D^{\mathrm{opt}}_{\alpha,\beta_\tau}
\subset
E_{\delta_0}(T_{\mathrm{mc}}).
\]
\end{cor}

\begin{proof}
Lemma~\ref{lem:optimal-independent-path} identifies the optimal
destabilizing subvarieties of the original pair with those of the
boundary pair $(\alpha,\beta_\tau)$, that is $\mathrm{Des}^{\mathrm{opt}}_{\alpha,\beta_\tau}(X)
=\mathrm{Des}^{\mathrm{opt}}_{\alpha,\beta}(X)$. Apply
Proposition~\ref{prop:optimal-contained-siu} to that boundary pair.
\end{proof}
\fi

\begin{thm}
Let \(X\) be a compact Kähler manifold and let \(\alpha,\beta\) be Kähler classes with \(\Gamma^{pp}_\beta(\alpha)\le0\). Let
$ V\in\operatorname{Des}^{\rm opt}_{\alpha,\beta}(X) $
be a smooth \(m\)-dimensional optimal destabilizing subvariety. Then the set of \((m-1)\)-dimensional optimally destabilizing subvarieties of
$ (V,\alpha|_V,\beta|_V) $
is finite.

 %Let $ \Gamma^{pp}_\be(\al)\leq 0 $ on a K\"ahler manifold  $M$  of arbitrary  dimension $n\geq2$, and $V\in\mathrm{Des}_{\alpha,\beta}^{\mathrm{opt}}(M) $ be smooth variety of dimension $m$.  Then for the K\"ahler manifold $(V,\al)$ and the pair $ (\al|_V,\be|_V) $ the collection of $m-1$ varieties in  $ \mathrm{Des}_{\alpha,\beta}^{\mathrm{opt}}(
 %   V) $ is finite.
\end{thm}
\begin{proof}
    By Lemma~\ref{lem:optimal-independent-path}, we have $\mathrm{Des}^{\mathrm{opt}}_{\alpha,\beta_\tau}(X)
=\mathrm{Des}^{\mathrm{opt}}_{\alpha,\beta}(X)$, we can assume that $ \Gamma^{pp}_\be(\al)= 0 $
    For any subvariety $V\subset X$ we define
\[
\mathrm{Des}_{\alpha,\beta}^{\mathrm{opt}}(X)|_V
=\{E\in \mathrm{Des}_{\alpha,\beta}^{\mathrm{opt}}(X):E\subset V\},
\]
Then for any subvariety $V\in \mathrm{Des}_{\alpha,\beta}^{\mathrm{opt}}(X)$, we have
\[
\mathrm{Des}_{\alpha,\beta}^{\mathrm{opt}}(V)=\mathrm{Des}_{\alpha,\beta}^{\mathrm{opt}}(X)|_V.
\] Indeed, this following from $\mu_{\al|V,\be|V}(E)=\mu_{\al,\be}(E).$
Let $ \dim V= m $. Now, we define the J-equation on the manifold $V$, using the restriction of the K\"ahler classes $\al$ and $\be$ from $X$, as follows $$\Lambda_{\omega_{\psi_t}} {\chi_t}=m\dfrac{\chi_t\wedge\omega_{\psi_t}^{m-1}}{\omega_{\psi_t}^m}=c,$$

where $\chi, \omega $ and $ \chi_t$ are the restriction of the metrics on the manifold $V$ defined as in the equation (\ref{eq:path}) defined in the beginning of Section \ref{subsec:linear} and $\psi_t$ is the solution at the time of $t.$ Then by Corollary \ref{cor:finite-optimal-divisors}, we can see that, the $m-1$ dimensional varieties in $\mathrm{Des}_{\alpha,\beta}^{\mathrm{opt}}(V)$ is finite.

\end{proof}

\subsection{Relation with degeneration of the continuity path}

The analytic set in
Proposition ~\ref{prop:optimal-contained-siu} is constructed from an
auxiliary mass-concentration current. It should be distinguished from
the actual locus where the solutions of the continuity path lose
regularity.

Let $\chi_t\in\beta_t$ and suppose that, for $t<\tau$, the equation
\[
\operatorname{tr}_{\omega_t}\chi_t=c,
\qquad
\omega_t\in\alpha,
\]
admits a smooth solution. Define the local smooth-compactness locus
\[
\mathcal R_\tau
:=
\left\{
x\in X:
\begin{array}{c}
\text{there exist a neighborhood }U\ni x
\text{ and a sequence }t_j\uparrow\tau\\
\text{such that }\omega_{t_j}
\text{ converges in }C^\infty_{\mathrm{loc}}(U)
\end{array}
\right\}.
\]

\begin{prop}
\label{prop:optimal-forces-nonregularity}
Assume that $ \Gamma^{pp}_\be(\al)\leq 0 $. Then one has an inclusion $V\subset X\setminus\mathcal{R}_\tau$, for any $V\in\mathrm{Des}_{\alpha,\beta}^{\mathrm{opt}}(X) $ and hence
\[
\mathcal D^{\mathrm{opt}}_{\alpha,\beta}
\subset
X\setminus\mathcal R_\tau.
\]
\end{prop}

\begin{proof}
By Lemma~\ref{lem:optimal-independent-path}, it is enough to consider
an irreducible $p$-dimensional optimal destabilizing subvariety $V$
for the boundary pair $(\alpha,\beta_\tau)$.

Suppose that some $x\in V$ belongs to $\mathcal R_\tau$. Then there
is a neighborhood $U\ni x$ and a sequence $t_j\uparrow\tau$ such
that
\(
\omega_{t_j}\longrightarrow\omega_\tau
\)
smoothly on compact subsets of $U$. Choose
\(
U'\Subset U
\)
with
\[
U'\cap V_{\mathrm{reg}}\neq\varnothing.
\]
Passing to the limit in the equation gives
\[
\operatorname{tr}_{\omega_\tau}\chi_\tau=c
\]
on $U$. And locally, this implies that $\chi_{t_j}\le c\omega_{t_j}$, hence $\omega_\tau\geq\frac{1}{c}\chi_\tau$. So, the limit remains K\"ahler on $U$.
Also at every point of $V_{\mathrm{reg}}\cap U$, the trace of $\chi_\tau$
on the proper $p$-plane $T^{1,0}V$ is strictly smaller than its full
$n$-dimensional trace. 

Therefore
$c\omega_\tau^p-p\chi_\tau\wedge\omega_\tau^{p-1}>0$
on $V_{\mathrm{reg}}\cap U$. It follows that
\begin{equation}\label{ineq:positive mass}
    \int_{U'\cap V_{\mathrm{reg}}}
\left(
c\omega_\tau^p
-
p\chi_\tau\wedge\omega_\tau^{p-1}
\right)>0.
\end{equation}

For every $t_j<\tau$, the same linear-algebra argument gives
\[
c\omega_{t_j}^p
-
p\chi_{t_j}\wedge\omega_{t_j}^{p-1}
>0
\]
on $V_{\mathrm{reg}}$. Hence
\[
\int_V\left(c\omega_{t_j}^p-
p\chi_{t_j}\wedge\omega_{t_j}^{p-1}
\right)\geq
\int_{U'\cap V_{\mathrm{reg}}}
\left(c\omega_{t_j}^p-p\chi_{t_j}\wedge\omega_{t_j}^{p-1}
\right).
\]
The right-hand side converges, by smooth convergence, to the strictly
positive number in \eqref{ineq:positive mass}.
On the other hand, the left-hand side is the cohomological quantity

${\int_V\left(c\alpha^p-p\beta_{t_j}\wedge\alpha^{p-1}\right)},$
which converges to
\[
\int_V
\left(
c\alpha^p
-
p\beta_\tau\wedge\alpha^{p-1}
\right)=0
\]
because $V$ is optimal for the boundary pair. This is a
contradiction.
\end{proof}

\begin{rem}
Proposition~\ref{prop:optimal-forces-nonregularity} says that every optimal destabilizing subvariety is contained in the locus where local smooth compactness fails.
No analyticity is asserted for $ X\setminus\mathcal R_\tau.$
The two loci introduced above therefore serve different purposes:
\[
E_{\delta_0}(T_{\mathrm{mc}})
\]
is a fixed auxiliary analytic set used for finiteness, whereas
\[
X\setminus\mathcal R_\tau
\]
records actual degeneration of the continuity path.
\end{rem}

\bigskip

\section{Rigidity of optimally destabilizing cycles}
\label{sec:rigidity}
\noindent In this section, we record a useful rigidity property of optimal
destabilizing subvarieties. The point is that, the optimal destabilizing subvarieties cannot move
in positive-dimensional families. Thus, any possible failure of finiteness
in higher dimension must come from isolated rigid cycles.

%Throughout this section we assume that $$\Gamma^{pp}_{\beta}(\alpha)=0.$$
%Thus, an optimal destabilizing subvariety is a proper irreducible analytic
%subvariety \(V\subsetneq X\) satisfying
%\(
%\mu_{\alpha,\beta}(V)=c.
%\)

%We shall distinguish carefully between the unconditional reverse Khovanskii--Teissier estimates and a stronger one-Rayleigh inequality, which is an additional hypothesis and does not hold for arbitrary nef triples according to \cite{HuXiao}. 

\subsection{A rigidity criterion in terms of a nef threshold}
\label{subsec:nef-threshold-rigidity}

In this subsection we record a rigidity consequence of the  weak reverse Khovanskii--Teissier inequality. The point is
that, the weaker coefficient-$2$ inequality is always available and gives useful rigidity provided the nef threshold of $\beta$ in
the direction of $\alpha$ is sufficiently large.

Let $X$ be a compact K\"ahler manifold of dimension $n$, and let
$\alpha,\beta$ be K\"ahler classes. Consider the following {\em nef threshold}
\[
s_\alpha(\beta)
:=
\sup\{s\in\mathbb R:\beta-s\alpha\ \text{is nef}\}.
\]
Since $\beta$ is K\"ahler, this number is positive. Moreover, if
$0<\varepsilon<s_\alpha(\beta)$, then $\beta-\varepsilon\alpha$ is
K\"ahler: indeed, choose $t>\varepsilon$ such that
$\beta-t\alpha$ is nef; then
\[
\beta-\varepsilon\alpha
=
(\beta-t\alpha)+(t-\varepsilon)\alpha
\]
is the sum of a nef class and a K\"ahler class.

We first record the following standard consequence of the Hodge index theorem in the precise form needed below (see \cite{LehmannXiao} for background).

\begin{lem}[Weak reverse Khovanskii--Teissier inequality]
\label{lem:weak-reverse-KT}
Let $Y$ be a compact K\"ahler manifold of dimension $m$. Let $A$ be
nef with $A^m>0$, and let $C,D$ be nef classes. Then
\[
(A^m)(C\cdot D\cdot A^{m-2})
\leq
2(C\cdot A^{m-1})(D\cdot A^{m-1}).
\]
\end{lem}

\begin{proof}
We recall the argument for the convenience of the reader.
By approximation, it is enough to prove the statement when $A$ is
K\"ahler. Set
\(
q(u,v):=u\cdot v\cdot A^{m-2}.
\)
By the Hodge index theorem, $q$ has signature $(1,*)$, and it is
negative definite on the primitive hyperplane
\(
\{u:q(u,A)=0\}.
\)
Let $ C_0=C-aA$ and $ D_0=D-bA $, where the constant $a$ and $b$ are chosen in such way to satisfy the conditions $q(C_0,A)=q(D_0,A)=0.$
Then
\[
a=\frac{C\cdot A^{m-1}}{A^m},
\qquad
b=\frac{D\cdot A^{m-1}}{A^m}.
\]
Since $C$ and $D$ are nef, we have $q(C,C)\geq0$ and $q(D,D)\geq0$.
Therefore
\[
-q(C_0,C_0)\leq a^2A^m,
\qquad
-q(D_0,D_0)\leq b^2A^m.
\]
By the Cauchy--Schwarz inequality for the negative definite form $-q$
on the primitive hyperplane,
\[
q(C_0,D_0)
\leq
\sqrt{(-q(C_0,C_0))(-q(D_0,D_0))}
\leq
abA^m.
\]
Hence
\[
q(C,D)
=
abA^m+q(C_0,D_0)
\leq
2abA^m.
\]
Substituting the values of $a$ and $b$ gives
\[
C\cdot D\cdot A^{m-2}
\leq
2\frac{(C\cdot A^{m-1})(D\cdot A^{m-1})}{A^m},
\]
which is the desired inequality.
\end{proof}

We now prove the threshold rigidity statement.

\begin{thm}
\label{thm:threshold-rigidity1}
Let $X$ be a compact K\"ahler manifold of dimension $n$, and let
$\alpha,\beta$ be K\"ahler classes.
Fix an integer $1\leq p\leq n-1$ and assume that
\[\label{ineq:nef threshold}
s_\alpha(\beta)>\Gamma^{\mathrm{pp}}_\beta(\alpha)+\frac{p-1}{p(p+1)}\bigl(c-n\Gamma^{\mathrm{pp}}_\beta(\alpha)\bigr).\tag{$\star$}
\]
Then the $p$-dimensional optimal subvarieties
of $X$ cannot move in a nonconstant compact one-parameter family. More precisely, there is no nonconstant analytic family \(\{V_t\}_{t\in T}\) over a connected compact complex curve \(T\) whose general member is an irreducible \(p\)-dimensional optimal subvariety.
\end{thm}

\begin{proof}
Suppose, for contradiction, that such a family exists. Thus we have a
compact connected curve $T$ and a nonconstant family of irreducible
$p$-dimensional subvarieties
$\{V_t\}_{t\in T}$ 
such that a general member $V_t$ is  optimal.
Let
\[
W:=\overline{\bigcup_{t\in T}V_t}
\]
be the sweep-out. Since the family is nonconstant,
$\dim W=p+1.$
In particular, if $p\leq n-2$ then $W$ is a proper subvariety of
$X$, while if $p=n-1$ then necessarily $W=X$.

Put
\[
\theta
:=
\Gamma^{\mathrm{pp}}_\beta(\alpha)+
\frac{p-1}{p(p+1)}
\bigl(c-n\Gamma^{\mathrm{pp}}_\beta(\alpha)\bigr).
\]
By hypothesis, $s_\alpha(\beta)>\theta$. Choose $ 0<\varepsilon<s_\alpha(\beta)$ 
such that
\begin{equation}\label{ineq:theta bound}
    \varepsilon>\theta.
\end{equation}
Let  $\beta_\varepsilon:=\beta-\varepsilon\alpha, $ 
then $\beta_\varepsilon$ is K\"ahler, hence nef.
Let
\[
f:Y\longrightarrow T,
\qquad
g:Y\longrightarrow W\subset X
\]
be a resolution of the universal family. Thus $Y$ is smooth compact
K\"ahler, $f$ has general fiber mapping generically finitely onto
$V_t$, and $g$ is generically finite onto $W$. Let $\ell$ be a
K\"ahler class on $T$, and set
\[
A:=g^*\alpha,
\qquad
B_\varepsilon:=g^*\beta_\varepsilon,
\qquad
L:=f^*\ell.
\]
The classes $A,B_\varepsilon,L$ are nef, and
$A^{p+1}>0$ because $g$ is generically finite onto the $(p+1)$-dimensional
subvariety $W$.
The standard degree-cancellation formulas give
\begin{equation}\label{eq:V family threshold}
    \mu_{\alpha,\beta_\varepsilon}(V_t)=p\frac{B_\varepsilon\cdot L\cdot A^{p-1}}{L\cdot A^p},
\end{equation}

and
\begin{equation}\label{eq:W threshold}
    \mu_{\alpha,\beta_\varepsilon}(W)=(p+1)\frac{B_\varepsilon\cdot A^p}{A^{p+1}}.
\end{equation}

Indeed, intersecting with $L=f^*\ell$ restricts the intersection
calculation to a general fiber of $f$, while the generically finite
degree of the map to $V_t$ cancels from numerator and denominator.
Similarly, the generically finite degree of $g:Y\to W$ cancels in
the slope computation for $W$.

If we apply Lemma~\ref{lem:weak-reverse-KT} on the $(p+1)$-fold $Y$ with
$
m=p+1, C=B_\varepsilon,$ and $ D=L,$ then we obtain
\[
(A^{p+1})(B_\varepsilon\cdot L\cdot A^{p-1})
\leq
2(B_\varepsilon\cdot A^p)(L\cdot A^p).
\]
Equivalently,
\[
(p+1)\frac{B_\varepsilon\cdot A^p}{A^{p+1}}
\geq
\frac{p+1}{2}
\frac{B_\varepsilon\cdot L\cdot A^{p-1}}
     {L\cdot A^p}.
\]
Using \eqref{eq:V family threshold} and \eqref{eq:W threshold}, this becomes
\begin{equation}\label{ineq:W and V threshold}
    \mu_{\alpha,\beta_\varepsilon}(W)\geq \frac{p+1}{2p} \mu_{\alpha,\beta_\varepsilon}(V_t).
\end{equation}

Since $V_t$ is optimal for the original class
$\beta$, we have
\begin{equation}\label{eq:perturbed V threshold}
    \mu_{\alpha,\beta}(V_t)=c-(n-p)\Gamma^{\mathrm{pp}}_\beta(\alpha).
\end{equation}

Because replacing $\beta$ by
$\beta_\varepsilon=\beta-\varepsilon\alpha$ subtracts
$d\varepsilon$ from the slope of every $d$-dimensional subvariety,
we get
\begin{equation}\label{eq:V family threshold1}
    \mu_{\alpha,\beta_\varepsilon}(V_t)=c-(n-p)\Gamma^{\mathrm{pp}}_\beta(\alpha)-p\varepsilon.
\end{equation}

Similarly,
\begin{equation}\label{eq:perturbed W threshold}
    \mu_{\alpha,\beta_\varepsilon}(W)=\mu_{\alpha,\beta}(W)-(p+1)\varepsilon.
\end{equation}

Combining \eqref{ineq:W and V threshold}, \eqref{eq:perturbed V threshold}, and \eqref{eq:perturbed W threshold}, we find
\[
\mu_{\alpha,\beta}(W)-(p+1)\varepsilon
\geq
\frac{p+1}{2p}
\bigl(c-(n-p)\Gamma^{\mathrm{pp}}_\beta(\alpha)-p\varepsilon\bigr).
\]
Therefore
\begin{equation}\label{ineq:weak W threshold}
    \mu_{\alpha,\beta}(W) \geq\frac{p+1}{2p}\bigl(c-(n-p)\Gamma^{\mathrm{pp}}_\beta(\alpha)\bigr)+\frac{p+1}{2}\varepsilon.
\end{equation}

We now distinguish two cases: $p\leq n-2$ and $p=n-1$.

Suppose first that $p\leq n-2$. Then $W\subsetneq X$, and by
definition 
\[
\Gamma^{\mathrm{pp}}_\beta(\alpha)
\leq
\frac{c-\mu_{\alpha,\beta}(W)}
     {n-p-1}.
\]
Equivalently,
\begin{equation}\label{ineq: W threshold}
    \mu_{\alpha,\beta}(W)\leq c-(n-p-1)\Gamma^{\mathrm{pp}}_\beta(\alpha).
\end{equation}

On the other hand, a direct calculation gives
\[  c-(n-p-1)\Gamma^{\mathrm{pp}}_\beta(\alpha)
-\frac{p+1}{2p}
\bigl(c-(n-p)\Gamma^{\mathrm{pp}}_\beta(\alpha)\bigr)
=\frac{p+1}{2}\theta.
\]
Since $\varepsilon>\theta$ by \eqref{ineq:theta bound}, inequality \eqref{ineq:weak W threshold} implies
\[
\mu_{\alpha,\beta}(W)
>
c-(n-p-1)\Gamma^{\mathrm{pp}}_\beta(\alpha),
\]
contradicting \eqref{ineq: W threshold}.

It remains to consider $p=n-1$. In this case $W=X$, and \eqref{ineq:weak W threshold} becomes
\begin{equation}\label{ineq:c and threshold}
    c=\mu_{\alpha,\beta}(X)\geq\frac{n}{2(n-1)}(c-\Gamma^{\mathrm{pp}}_\beta(\alpha))+\frac{n}{2}\varepsilon.
\end{equation}

But now
\[
c-
\frac{n}{2(n-1)}(c-\Gamma^{\mathrm{pp}}_\beta(\alpha))
=
\frac{n}{2}
\left[
\Gamma^{\mathrm{pp}}_\beta(\alpha)+
\frac{n-2}{n(n-1)}
(c-n\Gamma^{\mathrm{pp}}_\beta(\alpha))
\right]
=
\frac{n}{2}\theta.
\]
Since $\varepsilon>\theta$, the right-hand side of \eqref{ineq:c and threshold} is strictly
larger than $c$, a contradiction.

Thus no such nonconstant compact one-parameter family exists.
\end{proof}

\begin{rem}
The theorem is often most transparent after the linear normalization
\[
\beta_0:=\beta-\Gamma^{\mathrm{pp}}_\beta(\alpha)\alpha.
\]
Numerically,
\[
\mu_{\alpha,\beta_0}(Z)
=
\mu_{\alpha,\beta}(Z)-(\dim Z)\Gamma^{\mathrm{pp}}_\beta(\alpha),
\]
and hence
\[
\mu_{\alpha,\beta_0}(X)=c-n\Gamma^{\mathrm{pp}}_\beta(\alpha).
\]
Moreover, for every proper positive-dimensional subvariety
$Z\subset X$,
\[
\frac{\mu_{\alpha,\beta_0}(X)-\mu_{\alpha,\beta_0}(Z)}{n-\dim Z}=
\frac{c-\mu_{\alpha,\beta}(Z)}{n-\dim Z}
-\Gamma^{\mathrm{pp}}_\beta(\alpha).
\]
Thus the same subvarieties are optimal, but the infimum has been
shifted to zero:
\[
\Gamma^{\mathrm{pp}}_{\beta_0}(\alpha)=0
\]
as a numerical statement. Also
$ s_\alpha(\beta_0)=s_\alpha(\beta)-\Gamma^{\mathrm{pp}}_\beta(\alpha). $
Therefore, the hypothesis \eqref{ineq:nef threshold} is equivalent to
\[
s_\alpha(\beta_0)
>
\frac{p-1}{p(p+1)}
\mu_{\alpha,\beta_0}(X).
\]
This is the $\Gamma^{\mathrm{pp}}_\be(\al)=0$ version of the threshold
criterion. The proof above does not require $\beta_0$ itself to be
K\"ahler; the K\"ahler class used in the argument is
\(
\beta_\varepsilon=\beta-\varepsilon\alpha
\)
with $0<\varepsilon<s_\alpha(\beta)$.
\end{rem}

\begin{rem}
When $p\leq n-2$, the proof shows more explicitly that a moving
family of $p$-dimensional optimal
subvarieties would force its sweep-out $W$ to satisfy
\[
\mu_{\alpha,\beta}(W)
>
c-(n-p-1)\Gamma^{\mathrm{pp}}_\beta(\alpha),
\]
or equivalently
\[
\frac{c-\mu_{\alpha,\beta}(W)}{n-p-1}
<\Gamma^{\mathrm{pp}}_\beta(\alpha).
\]
This contradicts the defining property of
$\Gamma^{\mathrm{pp}}_\beta(\alpha)$.

When $p=n-1$, the sweep-out is the whole manifold $X$, and the same
intersection-theoretic estimate instead forces
\(
\mu_{\alpha,\beta}(X)>c,
\)
which directly contradicts the definition
$c=\mu_{\alpha,\beta}(X)$.
\end{rem}

\begin{rem}[Relation with adding multiples of $\alpha$]
The criterion is invariant under replacing $\beta$ by
$\beta+t\alpha$. Indeed,
\[
s_\alpha(\beta+t\alpha)=s_\alpha(\beta)+t, \mu_{\alpha,\beta+t\alpha}(X)=c+nt, 
\]
and
\[
\Gamma^{\mathrm{pp}}_{\beta+t\alpha}(\alpha)=\Gamma^{\mathrm{pp}}_\beta(\alpha)+t.
\]
Thus
\[
s_\alpha(\beta+t\alpha)
-
\Gamma^{\mathrm{pp}}_{\beta+t\alpha}(\alpha)
=
s_\alpha(\beta)-\Gamma^{\mathrm{pp}}_\beta(\alpha),
\]
and
\[
\mu_{\alpha,\beta+t\alpha}(X)
-
n\Gamma^{\mathrm{pp}}_{\beta+t\alpha}(\alpha)
=
c-n\Gamma^{\mathrm{pp}}_\beta(\alpha).
\]
As a consequence, the threshold condition depends only on the
boundary-normalized class $\beta-\Gamma^{\mathrm{pp}}_\beta(\alpha)\alpha$ and 
adding a large multiple of $\alpha$ only shifts all destabilizing
quotients by the same constant.
\end{rem}

\subsection{Applications of rigidity}
\label{subsec:applications-rigidity}

We now record some consequences of
Theorem~\ref{thm:threshold-rigidity}. The first two statements are
naturally formulated dimension by dimension. Thus, fix
$1\leq p\leq n-1$
and assume that
\[
s_\alpha(\beta)
>
\Gamma^{\mathrm{pp}}_\beta(\alpha) +
\frac{p-1}{p(p+1)}
\bigl(c-n\Gamma^{\mathrm{pp}}_\beta(\alpha)\bigr).
\tag{$\star_p$}
\]
At the J-semistable boundary $\Gamma^{\mathrm{pp}}_\beta(\alpha)=0$, this becomes
\[
s_\alpha(\beta)>
\frac{p-1}{p(p+1)}c.
\]
Since $ \frac{p-1}{p(p+1)}\leq\frac16$
for every $1\leq p\leq n-1$, the single condition
\[
s_\alpha(\beta)>\frac{c}{6}
\tag{$\star_{\mathrm{all}}$}
\]
implies $(\star_p)$ simultaneously in every dimension at the
boundary.

\subsubsection{Projective cycle-rigidity}

\begin{cor}[Projective cycle-rigidity]
\label{cor:projective-cycle-rigidity}
Let $X$ be projective, and suppose that $(\star_p)$ holds. Then every
$p$-dimensional optimal subvariety is isolated
in the Chow variety among irreducible reduced $p$-dimensional cycles.
\end{cor}

\begin{proof}
Fix a projective embedding of $X$, and let $d$ be the degree of the
subvariety under consideration. Denote by
\[
\operatorname{Chow}^{\mathrm{irr}}_{p,d}(X)
\subset
\operatorname{Chow}_{p,d}(X)
\]
the locus parametrizing irreducible reduced cycles.

Suppose that a $p$-dimensional optimal subvariety $V$ is not isolated in
$\operatorname{Chow}^{\mathrm{irr}}_{p,d}(X)$. Since the Chow variety
is projective, algebraic curve selection gives an integral projective
curve
\(
C\subset\operatorname{Chow}_{p,d}(X)
\)
through $[V]$ whose general point represents an irreducible reduced
cycle. After replacing $C$ by its normalization, the universal cycle
over the Chow variety restricts to an analytic family
\(
\{V_t\}_{t\in C}
\)
over a smooth compact curve.

The fundamental homology class of the members of a connected family
of cycles is constant. Hence
\[
\mu_{\alpha,\beta}(V_t)
=
\mu_{\alpha,\beta}(V)
\]
for every $t\in C$. In particular, the general member $V_t$ is again
 optimal. The family is nonconstant, contrary
to Theorem~\ref{thm:threshold-rigidity}.
\end{proof}

\subsubsection{Connected automorphism groups}

\begin{cor}[Automorphism invariance]
\label{cor:automorphism-invariance}
Let $X$ be projective, suppose that $(\star_p)$ holds, and let
\(
G\subset\operatorname{Aut}^0(X)
\)
be a connected algebraic subgroup. Then every $p$-dimensional
optimal subvariety is $G$-invariant.
\end{cor}

\begin{proof}
Since $G$ is connected, it acts trivially on cohomology. Indeed, the
homomorphism
\[
G\longrightarrow
\operatorname{GL}\bigl(H^\bullet(X,\mathbb Z)\bigr)
\]
is continuous, whereas its target is discrete, and is therefore
constant. Consequently,
\(
g^*\alpha=\alpha,
\)
and
\(
g^*\beta=\beta
\)
for every $g\in G$.

Let $V$ be a $p$-dimensional optimal subvariety. For every $g\in G$,
\[
\begin{aligned}
\mu_{\alpha,\beta}(gV)
&=
p\frac{\int_{gV}\beta\wedge\alpha^{p-1}}
        {\int_{gV}\alpha^p}
\\
&=
p\frac{\int_Vg^*\beta\wedge(g^*\alpha)^{p-1}}
        {\int_V(g^*\alpha)^p}
=
\mu_{\alpha,\beta}(V).
\end{aligned}
\]
Thus $gV$ is again optimal.

The orbit
\(
G\cdot[V]
\)
is a connected algebraic subset of the relevant Chow variety. If it
were positive-dimensional, then $[V]$ would not be isolated among
irreducible reduced cycles, contradicting
Corollary~\ref{cor:projective-cycle-rigidity}. Hence the orbit is a
single point, and therefore
\(
gV=V
\)
for every $g\in G$.
\end{proof}

\begin{rem}[No invariance assumption on the K\"ahler classes]
\label{rem:no-invariance-assumption}
No additional $G$-invariance assumption on $\alpha$ or $\beta$ is
needed in Corollary~\ref{cor:automorphism-invariance}. Their
invariance in cohomology follows automatically from the connectedness
of $G$. In particular, in the toric application below, neither the
K\"ahler classes nor their chosen representatives need to be assumed
torus-invariant. The argument uses only the cohomology classes and
the associated numerical slopes.
\end{rem}

\subsubsection{Toric varieties}

\begin{cor}[Toric varieties]
\label{cor:toric-rigidity}
Let $X$ be a smooth projective toric manifold with dense algebraic
torus
\(
T\cong(\mathbb C^*)^n.
\)
If $(\star_p)$ holds, then there are only finitely many
$p$-dimensional optimal subvarieties.
In particular, if
\(
\Gamma^{\mathrm{pp}}_\beta(\alpha)=0
\)
and
$
s_\alpha(\beta)>\frac{c}{6},
$
then the set of all optimal destabilizing subvarieties of $X$ is
finite.
\end{cor}

\begin{proof}
The torus $T$ is connected, and its action on $X$ gives a connected
algebraic subgroup of $\operatorname{Aut}^0(X)$. By
Corollary~\ref{cor:automorphism-invariance}, every
$p$-dimensional optimal subvariety is
$T$-invariant.

The irreducible $T$-invariant subvarieties of a toric variety are
precisely the closures of torus orbits. Equivalently, they are the
subvarieties associated with the cones of the defining fan. Since
the fan contains only finitely many cones, there are only finitely
many such subvarieties.
At the J-semistable boundary, the inequality
\(
s_\alpha(\beta)>\frac{c}{6}
\)
implies $(\star_p)$ for every $1\leq p\leq n-1$, proving the final
assertion.
\end{proof}

\subsubsection{Abelian varieties}

\begin{cor}[Abelian varieties]
\label{cor:abelian-rigidity}
Let $X$ be an abelian variety. If $(\star_p)$ holds, then $X$ admits
no proper $p$-dimensional optimal subvariety.

In particular, if $ \Gamma^{\mathrm{pp}}_\beta(\alpha)=0 $
and
$ s_\alpha(\beta)>\frac{c}{6},$
then $X$ admits no proper optimal destabilizing subvariety.
\end{cor}

\begin{proof}
The abelian variety $X$ acts on itself by translations. This is a
connected algebraic subgroup of $\operatorname{Aut}^0(X)$, and hence
Corollary~\ref{cor:automorphism-invariance} implies that every
$p$-dimensional optimal subvariety $V$ would
be invariant under every translation.

Choose a point $v\in V$. Given any $x\in X$, translation by $x-v$
maps $v$ to $x$. Translation-invariance of $V$ would therefore imply
$x\in V$. Hence $V=X$, contradicting the assumption that $V$ is
proper.

As before, the final assertion follows because
$s_\alpha(\beta)>c/6$ implies $(\star_p)$ in every dimension at the
J-semistable boundary.
\end{proof}

\begin{rem}
Thus, whenever the relevant nef-threshold condition holds,
optimal subvarieties are not merely
numerically distinguished: in the projective setting they are
isolated cycles and are fixed by every connected algebraic group
preserving the underlying geometric data. In particular, under the
uniform boundary condition $(\star_{\mathrm{all}})$, any failure of
finiteness would have to come from infinitely many isolated
subvarieties rather than from a moving family.
\end{rem}

\end{document}